\documentclass{article}
\usepackage{amsfonts,amsmath, amssymb,latexsym}
\def\Z{{\mathbb Z}}

\def\SL{{\rm SL}}
\def\GL{{\rm GL}}
\def\ct{{\rm ct}}
\def\Stab{{\rm Stab}}
\def\Tr{{\rm Tr}}

\def\Cl{{\rm Cl}}

\def\P{{\mathbb P}}
\def\Disc{{\rm Disc}}

\def\Aut{{\rm Aut}}\def\irr{{\rm irr}}
\def\Vol{{\rm Vol}}
\def\R{{\mathbb R}}
\def\F{{\mathbb F}}
\def\FF{{\mathcal F}}
\def\RR{{\mathcal R}}
\def\Q{{\mathbb Q}}

\def\U{{\mathcal U}}
\def\W{{\mathcal W}}
\def\V{{\mathcal V}}
\def\ZZ{{\mathcal Z}}
\def\Z{{\mathbb Z}}
\def\P{{\mathbb P}}
\def\F{{\mathbb F}}
\def\Q{{\mathbb Q}}
\def\C{{\mathbb C}}

\def\m1{{\textrm{\textbf{Id}}}}
\def\L{{\mathcal L}}

\newtheorem{theorem}{Theorem}

\newtheorem{lemma}[theorem]{Lemma}

\newtheorem{proposition}[theorem]{Proposition}
\newenvironment{proof}{\noindent {\bf Proof:}}{$\Box$ \vspace{2 ex}}

\title{On the Davenport-Heilbronn theorems and second order terms}

\author{Manjul Bhargava, Arul Shankar, and Jacob Tsimerman}

\begin{document}

\maketitle

\section{Introduction}

The classical theorems of Davenport and Heilbronn~\cite{DH} provide
asymptotic formulae for the number of cubic fields having bounded
discriminant and for the total number of 3-torsion elements in the class
groups of quadratic fields having bounded discriminant.  Specifically,
the theorems state:
\begin{theorem}[Davenport--Heilbronn]\label{DHth1}
  Let $N_{3}(\xi,\eta)$ denote the number of cubic fields $K$, up to
  isomorphism, that satisfy $\xi<\Disc(K)<\eta$. Then
\begin{equation}
\begin{array}{rcrcrcr}
{N_{3}(0,X)}&=&\displaystyle{\frac{1}{12\zeta(3)} X} &+& o(X)\:\!;\\[.125in]
N_{3}(-X,0)&=&\displaystyle{\frac{1}{4\zeta(3)}\, X} &+& o(X)\:\!.
\end{array}\end{equation}
\end{theorem}

\begin{theorem}[Davenport--Heilbronn]\label{DHth2}
Let $D$ denote the discriminant of a quadratic field and let
  $\Cl_3(D)$ denote the $3$-torsion subgroup of the ideal class group
  $\Cl(D)$ of $D$. Then
\begin{equation}
\begin{array}{rclcrcr}
\displaystyle{\displaystyle\sum_{0 < D < X} \#\Cl_3(D)} &=&
\!\displaystyle{\frac43}\cdot{\,\!\displaystyle\sum_{0 < D < X}\, 1}&+&o(X)\:\!;\\[.25in]
\displaystyle{\displaystyle\sum_{-X < D < 0} \#\Cl_3(D)} &=&
\displaystyle{2}\cdot{\!\displaystyle\sum_{-X < D < 0}\! 1}&+&o(X)\:\!.
\end{array}
\end{equation}
\end{theorem}
The Davenport--Heilbronn theorems, and the methods underlying their
proofs, have seen applications in numerous~works (see, e.g.,
\cite{Belabas1}, \cite{BF}, \cite{dodqf}, \cite{BW}, \cite{DK},
\cite{FK}, \cite{Vatsal}, \cite{Wong}).

Subsequent to their 1971 paper, extensive computations were undertaken
by a number of authors (see, e.g., Llorente--Quer~\cite{LQ} and
Fung--Williams~\cite{FW}) in an attempt to numerically verify the
Davenport--Heilbronn theorems.  However, computations up to
discriminants even as large as $10^7$ were found to agree quite poorly
with these theorems.  This in turn led to questions about the magnitude
of the error term in these theorems, and the problem of determining
precise second order terms.

In a related work, Belabas~\cite{Belabas1} developed a very fast method for
enumerating cubic fields---indeed, in essentially linear time
with the discriminant---allowing him to make tables of cubic fields up
to absolute discriminant $10^{11}$.  These computations still seemed
to agree rather poorly with the first Davenport--Heilbronn theorem, and led
Belabas to guess only the existence of error terms smaller than
$O(X/(\log\, X)^a)$ for any~$a$.
However, Belabas~\cite{Belabas2} later obtained the first
subexponential error terms for these theorems of the form
$O(X\exp(-\sqrt{\log
  X\log\log X}))$.

In 2000,
Roberts~\cite{Roberts} conducted a remarkable study of these latter
computations in conjunction with certain theoretical considerations,
which led him to conjecture a precise {\it second main term} for
Theorem~\ref{DHth1}.  This conjectural second main term took the form
of a certain explicit constant times $X^{5/6}$.
Further computations carried out in the last few years have revealed
Roberts' conjecture to agree extremely well with the data.  Meanwhile,
on the theoretical side, power-saving error terms for
Theorems~\ref{DHth1} and \ref{DHth2} were obtained by
Belabas, the first author, and Pomerance~\cite{BBP}, who showed
error terms of $O(X^{7/8+\epsilon})$.

The purpose of the current article is to prove the above conjecture of
Roberts.
More precisely, we prove the following theorem.

\begin{theorem}\label{main1}
Let $N_{3}(\xi,\eta)$ denote the number of cubic fields $K$, up to
isomorphism, that satisfy $\xi<\Disc(K)<\eta$.  Then
\begin{equation}
\begin{array}{rcrcrcr}
{N_{3}(0,X)}&\!=\!&\displaystyle{\frac{1}{12\zeta(3)} X} &\!+\!&
\displaystyle{\frac{4\zeta(1/3)}{5\Gamma(2/3)^3\zeta(5/3)} X^{5/6}}
&\!+\!& O_\epsilon(X^{5/6-1/48+\epsilon})\:\!;\\[.125in]
N_{3}(-X,0)&\!=\!&\displaystyle{\frac{1}{4\zeta(3)}\, X} &\!+\!&
\displaystyle{\frac{\sqrt{3}\cdot 4\zeta(1/3)}{5\Gamma(2/3)^3\zeta(5/3)} X^{5/6}}
&\!+\!& O_\epsilon(X^{5/6-1/48+\epsilon})\:\!.
\end{array}\end{equation}
\end{theorem}

Davenport and Heilbronn also proved a refined version of
Theorem~\ref{DHth1}, where they give the asymptotics for the number of
cubic fields $K$ having bounded discriminant satisfying any specified
set of splitting conditions at finitely many primes.  Roberts also
conjectures a precise second main term for the number of such fields
$K$ having discriminant bounded by $X$ (see \cite[Section~5]{Roberts}).  We also prove Roberts' refined conjecture in Section~9.

By essentially identical methods, we also prove the analogue of
Roberts' conjecture for the second Davenport--Heilbronn theorem, i.e.,
a precise second order term in Theorem~\ref{DHth2}.  Specifically,
we prove:

\begin{theorem}\label{main2}
Let $D$ denote the discriminant of a quadratic field and let
  $\Cl_3(D)$ denote the $3$-torsion subgroup of the ideal class group
  $\Cl(D)$ of $D$. Then
\begin{equation}
\begin{array}{rclcrcrcr}
\displaystyle{\displaystyle\sum_{0 < D < X} \#\Cl_3(D)} &=\;&
\!\displaystyle{\frac43}\cdot{\,\!\displaystyle\sum_{0 < D < X}\,
  1}&+&\displaystyle\frac{8\zeta(1/3)}{5\Gamma(2/3)^3}\prod_p\left(1-\frac{p^{1/3}+1}{p(p+1)}\right)\;X^{5/6} &+& O_\epsilon(X^{5/6-1/48+\epsilon})\:\!;\\[.275in]
\displaystyle{\displaystyle\sum_{-X < D < 0} \#\Cl_3(D)} &=\;&
\displaystyle{2}\cdot{\!\displaystyle\sum_{-X < D < 0}\! 1}&+&\displaystyle\frac{\sqrt3\cdot8\zeta(1/3)}{5\Gamma(2/3)^3}\prod_p\left(1-\frac{p^{1/3}+1}{p(p+1)}\right)\;X^{5/6}
&+&O_\epsilon(X^{5/6-1/48+\epsilon})\:\!.
\end{array}
\end{equation}
\end{theorem}

In the process, we present a simpler approach to proving the original
Davenport--Heilbronn theorems, and also a simpler approach to
establishing the theorem of Davenport~\cite{Davenport} on the density
of discriminants of binary cubic forms.  The second main term of the
latter theorem of Davenport (who obtained only a second term of
$O(X^{15/16})$) was first discovered by Shintani~\cite{Shintani} using
Sato and Shintani's theory of zeta functions for prehomogeneous vector
spaces~\cite{SatoShintani}.  In this article, we also give an
elementary derivation of this second main term of Shintani.  More
precisely, we prove:



\begin{theorem}[Davenport--Shintani]\label{bcfcount1}
Let $N(\xi,\eta)$ denote the number of $\GL_2(\Z)$-equivalence classes of
\linebreak irreducible integer-coefficient binary cubic forms $f$ satisfying
$\xi<\Disc(f)<\eta$.  Then
\begin{equation}
  \begin{array}{rcrcrcr}
    {N(0,X)}&\!=\!&\displaystyle{\frac{\pi^2}{72} X} &\!+\!&
    \displaystyle{\frac{\sqrt{3}\zeta(2/3)\Gamma(1/3)(2\pi)^{1/3}}{30\Gamma(2/3)}}X^{5/6}
    &\!+\!& O_\epsilon(X^{3/4+\epsilon})\:\!;\\[.125in]
    N(-X,0)&\!=\!&\displaystyle{\frac{\pi^2}{24} X} &\!+\!&
    \displaystyle{\frac{\zeta(2/3)\Gamma(1/3)(2\pi)^{1/3}}{10\Gamma(2/3)}} X^{5/6}
    &\!+\!& O_\epsilon(X^{3/4+\epsilon})\:\!.
\end{array}\end{equation}
\end{theorem}

In order to prove Theorems \ref{main1} and \ref{main2}, we need (in particular) to
apply a new, stronger version of Theorem~\ref{bcfcount1} where we
count equivalence classes of binary cubic forms satisfying any finite
or other suitable set of congruence conditions.  Such a theorem was
obtained by Davenport--Heilbronn but their method does not yield
second main terms.  Meanwhile, Shintani's zeta function method does
not immediately apply to cubic forms satisfying given congruence
conditions. We prove this congruence version of
Theorem~\ref{bcfcount1} in Section 6.

In fact, we use this more general version of Theorem~\ref{bcfcount1}
to prove a generalization of Theorems~\ref{main1} and~\ref{main2} that
also allows us to count cubic orders satisfying certain specified sets
of local conditions.  To state this more general theorem, we first
restate Theorem \ref{bcfcount1} as:
\begin{theorem}\label{ringwithres}
  Let $M_{3}(\xi,\eta)$ denote the number of isomorphism classes of
  orders $R$ in cubic fields that satisfy $\xi<\Disc(R)<\eta$.  Then
\begin{equation}
\begin{array}{rcrcrcr}
  {M_3(0,X)}&\!=\!&\displaystyle{\frac{\pi^2}{72} X} &\!+\!&
  \displaystyle{\frac{\sqrt{3}\zeta(2/3)\Gamma(1/3)(2\pi)^{1/3}}{30\Gamma(2/3)}}X^{5/6}
  &\!+\!& O_\epsilon(X^{3/4+\epsilon})\:\!;\\[.125in]
  M_3(-X,0)&\!=\!&\displaystyle{\frac{\pi^2}{24} X} &\!+\!&
  \displaystyle{\frac{\zeta(2/3)\Gamma(1/3)(2\pi)^{1/3}}{10\Gamma(2/3)}} X^{5/6}
  &\!+\!& O_\epsilon(X^{3/4+\epsilon})\:\!.
\end{array}\end{equation}
\end{theorem}

The proof of Theorem \ref{ringwithres} is relatively straightforward,
given Theorem~\ref{bcfcount1} and the ``Delone--Faddeev bijection''
between isomorphism classes of cubic orders and
$\GL_2(\Z)$-equivalence classes of irreducible binary cubic forms
(which we describe in more detail in Section~2).


The generalization of Theorems~\ref{main1} and~\ref{main2} (which will
also then include Theorem~\ref{ringwithres}) that we will prove allows
one to count cubic orders of bounded discriminant satisfying any
desired finite (or, in many natural cases, infinite) sets of local
conditions.  To state the theorem, for each prime $p$ let $\Sigma_p$
be any set of isomorphism classes of orders in \'etale cubic algebras
over $\Q_p;$ also, let $\Sigma_\infty$ denote any set of isomorphism
classes of \'etale cubic algebras over $\R$ $($i.e.,
$\Sigma_\infty\subseteq \{\R^3,\R\oplus\C\})$. We say that the
collection $(\Sigma_p)\cup\Sigma_\infty$ is {\it acceptable} if, for
all sufficiently large primes $p$, the set $\Sigma_p$ contains all
maximal cubic orders over $\Z_p$
that are not totally ramified.  We say that the collection
$(\Sigma_p)\cup\Sigma_\infty$ is {\it strongly acceptable} if, for all
sufficiently large primes~$p$, the set~$\Sigma_p$ consists of the set
of all cubic orders over~$\Z_p$, the set of all maximal
cubic orders over~$\Z_p$, or the set of all maximal cubic orders
over $\Z_p$ that are not totally ramified.

We wish to asymptotically count the total number of cubic orders $R$
of absolute discriminant less than $X$ that agree with such local
specifications, i.e., $R\otimes\Z_p\in\Sigma_p$ for all $p$ and
$R\otimes\R\in\Sigma_\infty$.  This asymptotic count---with the first
{\it two} main terms---is contained in the following theorem:

\begin{theorem}\label{gensigma1}
  Let $(\Sigma_p)\cup\Sigma_\infty$ be a strongly acceptable
  collection of local specifications, and let $\Sigma$ denote the set
  of all isomorphism classes of orders $R$ in cubic fields for which
  $R\otimes\Z_p\in\Sigma_p$ for all $p$ and
  $R\otimes\R\in\Sigma_\infty$.  For a free $\Z_p$-module $M$, define
  $M^{\rm Prim}\subset M$ by $M^{\rm Prim}:= M\backslash \{p\cdot
  M\}$.  Let $N_3(\Sigma;X)$ denote the number of cubic orders
  $R\in\Sigma$ that satisfy $|\Disc(R)|<X$.  Then
\begin{equation}
\begin{array}{rcl}
N_3(\Sigma;X)\!\!\!&=&\!\!\!
\displaystyle{\Bigl(\frac12\sum_{R\in\Sigma_\infty}
\frac1{|\Aut(R)|}\Bigr)\cdot
\prod_p\Bigl(\frac{p-1}{p}\cdot\sum_{R\in\Sigma_p}
\frac{1}{\Disc_p(R)}\cdot\frac1{|\Aut(R)|}\Bigr)}
\cdot X \vspace{.1in}\\[.1in]  &+&\,\,
\displaystyle\frac{1}{\zeta(2)}\displaystyle{\Bigl(\sum_{R\in\Sigma_\infty}\!\!
c_2(R)\Bigr)\cdot
\prod_p\Bigl((1-p^{-1/3})\cdot\sum_{R\in\Sigma_p}\frac1{\Disc_p(R)}\cdot\frac1{|\Aut(R)|}\int_{(R/\Z_p)^{{\rm Prim}}}i(x)^{2/3}dx}\Bigr)
\cdot X^{5/6}\,\,\vspace{.1in}\\
&+&\,\,O_{\epsilon}(X^{5/6-1/48+\epsilon})\:\!,
\end{array}\end{equation}
where $\Disc_p(R)$ denotes the discriminant of $R$ over $\Z_p$ as a
power of $p$, \, $i(x)$ denotes the index of $\Z_p[x]$~in~$R$, \,
$dx$ assigns measure $1$ to $(R/\Z_p)^{{\rm Prim}}$, and
$$c_2(R)=\begin{cases}
  \displaystyle{\frac{\sqrt{3}\zeta(2/3)\Gamma(1/3)(2\pi)^{1/3}}{30\Gamma(2/3)}}& \textrm{ if $R\cong\R\oplus\R\oplus\R$ }\vspace{.1in}\\
  \displaystyle{\,\,\;\;\;\;\frac{\,\zeta(2/3)\Gamma(1/3)(2\pi)^{1/3}}{10\Gamma(2/3)}}& \textrm{ if $R\cong\R\oplus\C.$ }\\
\end{cases}$$
\end{theorem}

Note that the case where $\Sigma_p$ consists of the maximal cubic
orders over $\Z_p$ for all $p$ yields Theorem~3, and also yields a
corresponding interpretation of the asymptotic constants in Theorem~3
as a product of local Euler factors.  Indeed, these Euler factors
correspond to local weighted counts of the possible cubic algebras
that can arise over $\Q_p$ and over $\Q_\infty=\R$.  Theorem~4 is deduced
by letting $\Sigma_p$ consist of all maximal cubic orders over $\Z_p$ that are not totally
ramified at $p$, and then applying class field theory (see \S8.1 and \S8.5).

Meanwhile, the
case where $\Sigma_p$ consists of {all} orders in \'etale cubic
algebras over $\Q_p$ yields Theorem~\ref{ringwithres}, and again also
yields the analogous interpretation of the constants in
Theorem~\ref{ringwithres}.  Theorem~\ref{gensigma1} thus
simultaneously generalizes Theorems~\ref{main1}, \ref{main2}, 5, and \ref{ringwithres}
in a natural way, and moreover, it yields a natural interpretation of
the various constants $\large{\frac{\pi^2}{72}}$, $\large{
  \frac{\pi^2}{24}}$, $\large{\frac{1}{12\zeta(3)}}$,
$\large{\frac{1}{4\zeta(3)}}$, 4/3, 2, 
etc.\ that appear in the asymptotics of
these theorems.

If we are only interested in the first main term, then we have the
following stronger result:
\begin{theorem}\label{gensigmafmt}
  Let $(\Sigma_p)\cup\Sigma_\infty$ be an acceptable collection of
  local specifications, and let $\Sigma$ denote the set of all
  isomorphism classes of orders $R$ in cubic fields for which
  $R\otimes\Q_p\in\Sigma_p$ for all $p$ and
  $R\otimes\R\in\Sigma_\infty$.  Let $N_3(\Sigma;X)$ denote the number
  of cubic orders $R\in\Sigma$ that satisfy $|\Disc(R)|<X$.  Then
\begin{equation}
  N_3(\Sigma;X)=\,\,\,\,\,\,
  \displaystyle{\Bigl(\frac12\sum_{R\in\Sigma_\infty}
    \frac1{|\Aut(R)|}\Bigr)\cdot
    \prod_p\Bigl(\frac{p-1}{p}\cdot\sum_{R\in\Sigma_p}
    \frac{1}{\Disc_p(R)}\cdot\frac1{|\Aut(R)|}\Bigr)}
  \cdot X+o(X).
\end{equation}
\end{theorem}
The case where, for all $p$, the set $\Sigma_p$ consists of all
maximal cubic rings is Theorem~\ref{DHth1}, while the case where it
consists of all maximal cubic rings that are not totally ramified at
$p$ yields Theorem~\ref{DHth2}.

Our proofs of Theorems 1--\ref{gensigmafmt} and particularly
Theorem~\ref{gensigma1}, though perhaps similar in spirit to the original
arguments of Davenport and Heilbronn, involve a number of new ideas
and refinements both on the algebraic and the analytic side.
First, we begin in Sections 2 and 3 by giving a much shorter and more
elementary derivation of the ``Davenport--Heilbronn correspondence''
between maximal cubic orders and appropriate sets of binary cubic
forms.

Second, we obtain the main term of the asymptotics of
Theorem~\ref{bcfcount1} in Section~5 by counting points not in a
single fundamental domain, but on average in a continuum of
fundamental domains, using a technique of~\cite{dodpf}.
This leads, in particular, to a uniform treatment of the cases of
positive and negative discriminants.  It also leads directly to
stronger error terms; most notably,
we obtain immediately an error term of $O(X^{5/6})$ for the number of
$\GL_2(\Z)$-equivalence classes of integral binary cubic forms of
discriminant less than $X$, improving on Davenport's original
$O(X^{15/16})$.  The $O(X^{5/6})$ term is seen to come from the
``cusps'' or ``tentacles'' of the fundamental regions.

Third, to more efficiently count points in the cusps of these
fundamental regions, we introduce a ``slicing and smoothing''
technique in Section~6, which then allows us to keep track of precise
second order terms and thus also prove the second main term of
Theorem~\ref{bcfcount1}.  The technique works equally well when
counting points satisfying any finite set of congruence conditions
(see Theorem~\ref{shincong}).

Fourth, our use of the Delone--Faddeev correspondence (c.f.\
Section~2) allows us to give an elementary treatment of the analogue
of Theorem \ref{main1} for orders, rather than just fields, as in
Theorem~\ref{ringwithres} and the cases of Theorem~\ref{gensigma1}
where only finitely many local conditions are involved.  We
prove the main terms of Theorems~1--8 in Section 8, using a simplified
computation of $p$-adic densities that is carried out in Section~4.

Finally---in order to treat the second term in cases where infinitely
many local conditions are involved---we introduce a sieving method
that allows one to preserve the second main terms even when certain
natural infinite sets of congruence conditions are applied.  This is 
accomplished in Section 9, using a computation of ``second order
$p$-adic densities'' that is carried out in Section~7.

\vspace{.1in}
\noindent {\bf Remark 1.} We note that an alternative proof of Theorems 3 and 4 has recently been obtained by Taniguchi and Thorne~\cite{TT}, using quite different methods.  Although our proof here is more elementary, the work of Taniguchi--Thorne connects with the theory of Shintani zeta functions, and may thus
have further interesting consequences in that realm.  In fact, it seems clear that the methods here in conjunction with those of~\cite{TT} should together yield even stronger results, e.g., better error terms, than either method alone! We hope to pursue this in future work.

\vspace{.1in}
\noindent {\bf Remark 2.} Readers interested mainly in our new simpler
proofs of the main terms of the Davenport--Heilbronn theorems may
safely skip Sections~6, 7 and 9, which constitute about a half of
this paper.  On the other hand, those interested in the new results on
second main terms may wish to concentrate primarily on these sections.

\section{The Delone--Faddeev correspondence}

A {\it cubic ring} is any commutative ring with unit that is free of
rank 3 as a $\Z$-module.
We begin with a theorem of Delone--Faddeev~\cite{DF} (as refined by
Gan--Gross--Savin~\cite{GGS}) parametrizing cubic rings by
$\GL_2(\Z)$-equivalence classes of integral binary cubic forms.
Throughout this paper, we always use the ``twisted'' action of
$\GL_2(\Z)$ on binary cubic forms, i.e., an element
$\gamma\in\GL_2(\Z)$ acts on a binary cubic form $f(x,y)$ by
\begin{equation}\label{action}
(\gamma f)(x,y) = \frac1{\det(\gamma)}
f((x,y)\gamma).
\end{equation}

\begin{theorem}\label{df}{\em (\cite{DF},\cite{GGS})}
  There is a natural bijection between the set of
  $\GL_2(\Z)$-equivalence classes of integral binary cubic forms and
  the set of isomorphism classes of cubic rings.
\end{theorem}
\begin{proof}
  Given a cubic ring $R$, let $\langle 1,\omega,\theta\rangle$ be a
  $\Z$-basis for $R$.  Translating $\omega$ and $\theta$ by the appropriate
  elements of $\Z$, we may assume that $\omega\theta\in\Z$.  In
  the terminology of~\cite{DF}, a basis satisfying the latter
  condition is called {\it normal}.
If $\langle 1,\omega,\theta\rangle$ is a normal basis, then there exist constants
$a,b,c,d,\ell,m,n\in\Z$ such that
\begin{equation}\label{ringlaw3}
\begin{array}{cll}
  \omega\theta &=& n \\
  \omega^2 &=& m - b \omega + a \theta \\
  \theta^2 &=& \ell\, - d \omega + c \theta.
\end{array}
\end{equation}
To the cubic ring
$R$, we associate the binary cubic form $f(x,y)=ax^3+bx^2 y+cx y^2+dy^3$.

In more coordinate-free terms, the form $f(x,y)$ represents the cubic map
$R/\Z\to \wedge^2(R/\Z)\cong\Z$ given by $r\mapsto r\wedge r^2$.
To see this, set $r=x\omega+y\theta$; then
\[
r\wedge r^2 = (x\omega+y\theta)\wedge[x^2(b \omega - a \theta)
+y^2(d \omega - c \theta)]
= f(x,y)(\omega\wedge\theta)
\]
as elements of $\wedge^2(R/\Z)$. In particular, changing the
$\Z$-basis $\langle \omega,\theta\rangle$ of $R/\Z$ by an element
$\gamma\in\GL_2(\Z)$, and then renormalizing the basis in $R$,
transforms the corresponding binary cubic form $f(x,y)$ by that same
element of $\GL_2(\Z)$.

Conversely, given a binary cubic form $f(x,y)=ax^3+bx^2 y+cx
y^2+dy^3$, form a potential cubic ring having multiplication laws
(\ref{ringlaw3}).  The values of $\ell,m,n$ are subject to the
associative law relations
$(\omega\theta)\theta=\omega(\theta^2)$ and
$(\omega^2)\theta=\omega(\omega\theta)$, which when multiplied
out using (\ref{ringlaw3}), yield a system of equations which possesses a
unique solution for $n,m,\ell$, namely
\begin{equation}\label{ringasslaw3}
\begin{array}{rll}
n &=& -ad \\ m &=& -ac \\ \ell &=& -bd.
\end{array}
\end{equation}
If follows that any binary cubic form $f(x,y)=ax^3+bx^2y+cxy^2+dy^3$,
via the recipe (\ref{ringlaw3}) and (\ref{ringasslaw3}), leads to a
unique cubic ring $R=R(f)$.
This is the desired
conclusion.
\end{proof}

The map $f\mapsto R(f)$ has many desirable properties.  First, it is
{\it discriminant-preserving}.  More precisely, if $R$ is a cubic
ring, then we may define the {\it trace} $\Tr(\alpha)\in\Z$ of an element
$\alpha\in R$ as the trace of the $\Z$-linear mapping
$\times\alpha:R\to R$.
The {\it discriminant} $\Disc(R)$ of a cubic ring $R$ is then the
determinant of the bilinear pairing $\Tr(\alpha\beta)_{\alpha,\beta\in
  R}$ on $R$.  It turns out that this discriminant coincides with the
discriminant of the corresponding binary cubic form:

\begin{proposition}\label{disceq}
  The discriminant of an integral binary cubic form $f$
is equal to the discriminant of the corresponding cubic ring $R(f)$.
\end{proposition}

\begin{proof}
An explicit calculation using (\ref{ringlaw3}) and (\ref{ringasslaw3})
easily verifies Proposition~\ref{disceq}.  The proposition can also be
deduced more conceptually as follows.  We observe
that the discriminant of $R(f)$ must be an $\SL_2(\Z)$-invariant
polynomial in $a,b,c,d$ of degree~4.  It is well-known (see, e.g.,~\cite{Hilbert}) that a binary cubic form $f$ possesses, up to scaling,
only one $\SL_2(\Z)$-invariant polynomial of degree 4, namely the
discriminant $\Disc(f)$.  We conclude that
$\Disc(R(f))=c\cdot\Disc(f)$ for some constant~$c$.  To determine $c$,
let $f(x,y)=xy(x-y)$.  Then by (\ref{ringlaw3}), we have
$R(f)\cong\Z^3$ (with the identification $\omega\mapsto (-1,0,0)$ and
$\theta\mapsto (0,-1,0)$).  Since $\Disc(xy(x-y))=1$ with the usual normalization of the discriminant, and
$\Disc(R(f))=\Disc(\Z^3)=1$, we conclude that $c=1$.
\end{proof}

\noindent
Explicitly, the
discriminant of the binary cubic form $f$ (and thus of the
corresponding cubic ring $R(f)$) is given by
\begin{equation}
\Disc(R(f))=\Disc(f)= b^2c^2-4ac^3-4b^3d - 27a^2d^2 + 18abcd.
\end{equation}

Next, we may determine whether $R(f)$ is an integral domain simply
by checking the reducibility/irreducibility of $f$ over $\Q$:

\begin{proposition}\label{intdomain} For an integral binary cubic form $f$, the cubic
  ring $R(f)$ is an integral domain if and only if $f$ is irreducible
  as a polynomial over $\Q$.
\end{proposition}

\begin{proof}
  If $f(x,y)=ax^3+bx^2y+cxy^2+dy^3$ is reducible, then it has a linear
  factor, which (by a change of variable in $\GL_2(\Z)$) we may assume
  is $y$; i.e., $a=0$.  In this case, (\ref{ringlaw3}) and
  (\ref{ringasslaw3}) show that $\omega\theta=0$, so $R(f)$ has zero divisors.

Conversely, if a cubic ring $R$ has zero divisors, then there exists
some element $\omega\in R$ such that $\langle1,\omega\rangle$ spans a
quadratic subring of $R$.  Such an $\omega$ can be constructed as
follows.  Let $\alpha$ and $\beta$ be two nonzero elements of $R$ with
$\alpha\beta=0$, and let $\alpha^3+c_1\alpha^2+c_2\alpha+c_3=0$ be the
characteristic equation of the $\Z$-linear mapping $\times \alpha:R\to R$.
Multiplying both sides by
$\beta$, we see that $c_3=0$, so that
$\alpha(\alpha^2+c_1\alpha+c_2)=0$.  If $\alpha^2+c_1\alpha+c_2=0$,
then we may let $\omega=\alpha$.  Otherwise, note that
$(\alpha^2+c_1\alpha+c_2)^2=c_2(\alpha^2+c_1\alpha+c_2)$, so in that
case we may set $\omega=\alpha^2+c_1\alpha+c_2$, and
$\omega^2=c_2\omega$.  Either way, we see that
$\langle1,\omega\rangle$ spans a quadratic subring of $R$.

Scaling $\omega$ by an integer if necessary, we may assume that
$\omega$ is a primitive vector in the lattice $R\cong \Z^3$, and then
extend $\langle1,\omega\rangle$ to a basis
$\langle1,\omega,\theta\rangle$ of $R$.  Normalizing this basis if
needed, we then see in (\ref{ringlaw3}) that we must have $a=0$,
implying that the associated binary cubic form is reducible.  We
conclude that, under the Delone--Faddeev correspondence, integral
domains correspond to irreducible binary cubic forms.
\end{proof}

Other important properties of the cubic ring $R(f)$ can also be read off easily
from the binary cubic form $f$.  For example, we have

\begin{proposition}\label{aut}
{For an integral binary cubic form $f$, the group of ring
automorphisms of $R(f)$ is naturally isomorphic to the stabilizer of $f$ in
$\GL_2(\Z)$}.
\end{proposition}
\begin{proof}
  This follows directly from the proof of Theorem~\ref{df}: any
  automorphism of $R(f)$ results in a $\GL_2(\Z)$-transformation on
  the chosen normal basis $\omega,\theta$ of $R/\Z$ (which is then
  automatically still normal), thus giving an element of the
  stabilizer of the binary cubic form $f$ in $\GL_2(\Z)$; the converse
  is similarly trivial.
\end{proof}

Finally, we note that the correspondence of Theorem~\ref{df}, and the
analogues of Propositions~\ref{disceq}--\ref{aut}, also hold for cubic
algebras and binary cubic forms over other base rings such as $\C$,
$\R$, $\Q$, $\Q_p$, $\Z_p$, and $\F_p$.  Indeed, let~$T$ denote any one of
these rings. Then a {\it cubic ring over~$T$} can be defined
analogously as any ring with unit that is free of rank 3 as a $T$-module.
Similarly, a {\it binary cubic form over~$T$} is any binary cubic form
with coefficients in~$T$.  Again, $\GL_2(T)$ acts on the space of
binary cubic forms over~$T$ via~(\ref{action}).  With these
definitions, Theorem~\ref{df} and Propositions~\ref{disceq}--\ref{aut}
all hold when ``$\GL_2(\Z)$'' is replaced by ``$\GL_2(T)$'',
``integral binary cubic form'' is replaced by ``binary cubic form
over $T$'', and ``cubic ring'' is replaced by ``cubic ring over $T$'';
the proofs are identical.  This observation will also be very useful
to us in later sections.

\section{The Davenport--Heilbronn correspondence}


A cubic ring is said to be {\it maximal} if it is not a subring of any
other cubic ring.  The first part of the Davenport--Heilbronn
theorem~\cite{DH} describes a bijection (known as the
``Davenport--Heilbronn correspondence'') between maximal cubic rings
and certain special classes of binary cubic forms.  In this section,
we give a simple derivation of this bijection.

By the work of the previous section, in order to obtain the
Davenport--Heilbronn correspondence we must simply determine which
binary cubic forms $f$ yield maximal rings $R(f)$ in the bijection
given by (\ref{ringlaw3}) and (\ref{ringasslaw3}).  Now a cubic ring
$R$ is maximal if and only if the cubic $\Z_p$-algebra
$R_p=R\otimes\Z_p$ is maximal for every $p$ (this is because $R$ is a
maximal ring if and only if it is isomorphic to a product of rings of
integers in number fields).  The condition on $R$ that $R\otimes \Z_p$
be a maximal cubic algebra over $\Z_p$ is called ``maximality at
$p$''.  The following lemma illustrates the ways in which a ring $R$
can fail to be maximal at $p$:
\begin{lemma}\label{nonmax}
Suppose $R$ is not maximal at $p$. Then there is a $\Z$-basis
$\langle1,\omega,\theta\rangle$ of $R$ such
that at least one of the following is true:
\begin{itemize}
\item $\Z+\Z\cdot(\omega/p)+\Z\cdot\theta$ forms a ring
\item $\Z+\Z\cdot(\omega/p)+\Z\cdot(\theta/p)$ forms a ring.
\end{itemize}
\end{lemma}

\begin{proof}
  Let $R'\supset R$ be any ring strictly containing $R$ such that the
  index of $R$ in $R'$ is a multiple of $p$, and let $R_1=R'\cap
  (R\otimes_\Z \Z[\frac1p])$.
  Then the ring $R_1$ also strictly contains $R$, and the index of $R$
  in $R_1$ is a power of $p$.  By the theory of elementary divisors,
  there exist nonnegative integers $i\geq j$ and a basis $\langle
  1,\omega,\theta\rangle$ of $R$ such that
  \begin{equation}\label{basisR1} R_1 = \Z +
    \Z(\omega/p^i)+\Z(\theta/p^j). \end{equation} If $i=1$, we are
  done.  Hence we assume $i>1$.

  We normalize the basis $\langle 1,\omega,\theta\rangle$ if
  necessary; this does not affect the truth of equation
  (\ref{basisR1}).  Now suppose the multiplicative structure of $R$ is
  given by (\ref{ringlaw3}) and (\ref{ringasslaw3}).  That the right
  side of (\ref{basisR1}) is a ring translates into the following
  congruence conditions on $a,b,c,d$:\footnote{We follow here the
    convention that, for $e\leq 0$, we have $a\equiv 0$ (mod $p^e$)
    for any integer $a$.}
\begin{equation}\label{conds}
a\equiv 0 \!\!\pmod{p^{2i-j}}, \;\:b\equiv 0 \!\!\pmod{p^{i}}, \;\:
c\equiv 0 \!\!\pmod{p^{j}}, \;\:d\equiv 0 \!\!\pmod{p^{2j-i}}.
\end{equation}
If $j=0$, then replacing $(i,j)$ by $(i-1,j)$ maintains
the truth of the above congruences, and $R_1$ as defined by (\ref{basisR1})
remains a ring.  If $j>0$, then we may replace $(i,j)$
instead by $(i-1,j-1)$.  Thus in a finite sequence of such moves, we arrive
at $i=1$, as desired.
\end{proof}

The lemma implies that a cubic ring $R(f)$ can fail to be maximal at
$p$ in two ways: either (i) $f$ is a multiple of $p$, or (ii) there is some
$\GL_2(\Z)$-transformation of $f(x,y)=ax^3+bx^2y+cxy^2+dy^3$ such that
$a$ is a multiple of $p^2$ and $b$ is a multiple of $p$.

Let $\U_p$ be the set of all binary cubic forms $f$ not satisfying either
of the latter two conditions.  Then we have proven

\begin{theorem}\label{cubmax}{\em (Davenport--Heilbronn~\cite{DH})}
The cubic ring $R(f)$ is maximal at $p$ if and only if $f\in \U_p$.
The cubic ring $R(f)$ is maximal if and only if $f\in \U_p$ for all $p$.
\end{theorem}
Note that our definition of $\U_p$ is somewhat simpler than that
used by Davenport--Heilbronn (but is easily seen to be equivalent).

\vspace{.1in}
The discussion above can also be used to deduce a number of other
consequences.  For example, we may use it to determine the number of
index $p$ {subrings} of a given cubic ring $R(f)$ as well as the 
number of cubic rings {containing} a given cubic ring $R(f)$ with index $p$:

\begin{proposition}\label{subring}
  For an integral binary cubic form $f$, the number of index $p$
  subrings of $R(f)$ is equal to $\omega_p(f)$, the number of zeroes in
  $\P^1(\F_p)$ of $f$ modulo $p$.
\end{proposition}


\begin{proposition}\label{supring}
  For an integral binary cubic form $f$, the number of cubic rings in
  $R(f)\otimes \Q$ containing $R(f)$ with index $p$ is equal to the
  number of double zeroes $\alpha\in \P^1(\F_p)$ of $f$ modulo $p$
  such that $f(\alpha')\equiv 0$ $({\rm mod}$~$p^2)$ for all
  $\alpha'\equiv \alpha$ ${\rm mod}$~$p$.
\end{proposition}

\begin{proof}
If $R\subset R'$ with $[R':R]=p$, then we may write $R=\Z+pR'+\Z\theta$, where $\theta$ is 
a well-defined element of $(R'/\Z)/p(R'/\Z)$.   Extending $\theta$ to a $\Z$-basis $1,\omega,\theta$ of $R'$, and 
renormalizing if necessary, we see that $1,\omega,\theta$ is a $\Z$-basis for $R'$ and $1,p\omega,\theta$ is a $\Z$-basis
for $R$.  Regardless of these choices, note that $\theta$ is well-defined in $(R'/\Z)/p(R'/\Z)$, while $p\omega$ is well-defined in $(R/\Z)/p(R/\Z)$. 

Now if $f'(x,y)=a'x^3+b'x^2y+c'xy^2+d'y^3$ is the binary cubic form corresponding to the normal basis $1,\omega,\theta$ of the ring $R'$, then by (\ref{ringlaw3}) we see that $R=\Z+pR'+\Z\theta$ is also a ring if and only if $d'\equiv 0$ (mod $p$), i.e., the image of
$\theta$ in $R'/\Z$ is a root of $f'$ (mod $p$), when $f'$ is viewed as a cubic map $R'/\Z\to\wedge^2(R'/\Z)\cong\Z$ given by $r\mapsto r\wedge r^2$.  In that case, $f(x,y)=a'p^2x^3+b'px^2y+c'xy^2+(d'/p)y^3$ is the binary cubic form corresponding to the basis $1,p\omega,\theta$ of $R$, and this gives the desired bijection between roots of $f'$ (mod $p$) and subrings of $R'$ of index $p$, as stated in Proposition~\ref{subring}.

Similarly, if $f(x,y)=ax^3+bx^2y+cxy^2+dy^3$ is the binary cubic form corresponding to the normal basis $1,p\omega,\theta$ of the ring $R$, then by (\ref{ringlaw3}) we see that the $\Z$-module $R'$ spanned by $1,\omega,\theta$ is also a ring if and only if $a\equiv 0$ (mod~$p^2$) and $b\equiv 0$ (mod $p$), i.e., 
the image of $p\omega$ in $R/\Z$ is a double root of $f$ (mod~$p$) and $f$ takes a value at $p\omega$ that is a multiple of $p^2$, when $f$ is viewed as a cubic map $R/\Z\to\wedge^2(R/\Z)\cong\Z$ given by $r\mapsto r\wedge r^2$.  In that case, $f'(x,y)=(a/p^2)x^3+(b/p)x^2y+cxy^2+dpy^3$ is the binary cubic form corresponding to the basis $1,\omega,\theta$ of $R'$, and this gives the desired bijection between roots $\alpha$ of $f$ (mod $p$) such that $f(\alpha)\equiv 0$ (mod $p^2)$, and rings $R'$ containing $R$ with index $p$, as stated in Proposition~\ref{supring}.
\end{proof}
\section{Local behavior and $p$-adic densities}\label{secpden1}

In this section, we consider elements $f$ in the spaces of binary
cubic forms over the integers $\Z$, the $p$-adic ring $\Z_p$, and
the residue field $\Z/p\Z$.  We denote these spaces by $V_\Z$,
$V_{\Z_p}$, and $V_{\F_p}$ respectively.  The results in this section
are also contained in~\cite{DH}; however, we give here slightly simpler and more
direct proofs.

Aside from the degenerate case $f\equiv 0$ (mod $p$), any form $f\in
V_\Z$ (resp.\ $V_{\Z_p}$, $V_{\F_p}$) determines exactly three points
in $\P^1_{\bar\F_p}$, obtained by taking the roots of $f$ reduced
modulo $p$.  For such a form $f$, define the symbol $(f,p)$ by setting
\[ (f,p) = (f_1^{e_1} f_2^{e_2} \cdots), \] where the $f_i$'s indicate
the degrees of the fields of definition over $\F_p$ of the roots of
$f$, and the $e_i$'s indicate the respective multiplicities of these
roots.  There are thus five possible values of the symbol $(f,p)$,
namely, $(111)$, $(12)$, $(3)$, $(1^21)$, and $(1^3)$.  Furthermore,
it is clear that if two binary cubic forms $f_1,f_2$ over $\Z$ (resp.\
$\Z_p$, $\F_p$) are equivalent under a transformation in $\GL_2(\Z)$
(resp.  $\GL_2(\Z_p)$, $\GL_2(\F_p))$, then $(f_1,p)=(f_2,p)$. By
$T_p(111), T_p(12)$, etc., let us denote the set of $f$ such that
$(f,p)=(111)$, $(f,p)=(12)$, etc.

By the definition of $R(f)$, the ring structure of the quotient ring
$R(f)/(p)$ depends only on the $\GL_2(\F_p)$-orbit of $f$ modulo $p$;
hence the symbol $(f,p)$ indicates something about the structure of
the ring $R(f)$ when reduced modulo $p$.  In fact, writing down the
multiplication laws at one point of each of the five aforementioned
$\GL_2(\F_p)$-orbits demonstrates that
\[
(f,p)= (f_1^{e_1} f_2^{e_2} \cdots) \iff
  R(f)/(p) \cong \F_{p^{f_1}}[t_1]/(t_1^{e_1})
            \oplus\F_{p^{f_2}}[t_2]/(t_2^{e_2})\oplus\cdots.\]
In particular, it follows that for $f\in \U_p$, the symbol $(f,p)$
conveys precisely the splitting behavior of $R(f)$ at $p$.  For example,
if $(f,p)=(1^3)$ for $f\in \U_p$, then this means the maximal cubic
ring $R(f)$ is totally ramified at $p$.

Now, for any set $S$ in $V_\Z$ (resp.\ $V_{\Z_p}$, $V_{\F_p}$) that is
definable by congruence conditions, let us denote by $\mu(S)=\mu_p(S)$
the $p$-adic density of the $p$-adic closure of $S$ in $V_{\Z_p}$,
where we normalize the additive measure $\mu$ on $V_{\Z_p}=\Z_p^4$ so
that $\mu(V_{\Z_p})=1$ (i.e., we have taken the product of the usual
additive measures on $\Z_p$).  The following lemma determines the
$p$-adic densities of the sets $T_p(\cdot)$.

\begin{lemma}
We have
\[
\begin{array}{rcl}
\mu(T_p(111))&=&\!\!\frac{1}{6}\,(p-1)^2\;p\;(p+1)\,/\,p^{4}\\[.045in]
\mu(T_p(12))&=& \!\!\frac{1}{2}\,(p-1)^2\;p\;(p+1)\,/\, p^{4} \\[.045in]
\mu(T_p(3))&=& \!\!\frac{1}{3}\,(p-1)^2\;p\;(p+1)\,/\, p^{4} \\[.045in]
\mu(T_p(1^21))&=& \,\;\;(p-1)\;\,p\;(p+1)\,/\, p^{4} \\[.045in]
\mu(T_p(1^3))&=& \,\;\;(p-1)\;\;\;(p+1)\,/\,p^{4}\;. \\[.045in]
\end{array}
\]
\end{lemma}

\begin{proof}
Since the criteria for membership of $f$ in a $T_p(\cdot)$
depend only on the residue class of $f$ modulo $p$, it suffices
to consider the situation over $\F_p$.
We examine first $\mu(T_p(111))$.  The number of unordered triples of
distinct points in $\P^1_{\F_p}$ is $\frac{1}{6}(p+1)p(p-1)$.
Furthermore, given such a triple of points, there is a unique binary
cubic form, up to scaling, having this triple of points as its roots.
Since the total number of binary cubic forms
over $\F_p$ is $p^{4}$, it follows that
${\textstyle
\mu(T_p(111))=\frac{1}{6}\bigl[(p+1)p(p-1)\bigr](p-1)/\,p^{4},}
$
as given by the lemma.

Similarly, the number of unordered triples of points, one member of
which is in $\P^1_{\F_p}$ while the other two are $\F_p$-conjugate in
$\P^1_{\F_{p^2}}$, is given by $\frac{1}{2}(p+1)(p^2-p)$.  We thus have
${\textstyle
  \mu(T_p(12))=\frac{1}{2}\bigl[(p+1)(p^2-p)\bigr](p-1)/\,p^{4}.}$
Also, the number of unordered $\F_p$-conjugate triples of
distinct points in $\P^1_{\F_{p^3}}$ is $(p^3-p)/3$, and hence
${\textstyle \mu(T_p(3))=\frac13\bigl[(p^3-p)\bigr](p-1)/\,p^{4}.}$

Meanwhile, the number of pairs $(x,y)$ of distinct points in
$\P^1_{\F_p}$ is given by $(p+1)p$, so that the number of binary cubic
forms over $\F_p$ having a double root at some point $x$ and a single
root at another point $y$ is $[(p+1)p](p-1)$.  Thus ${\textstyle
  \mu(T_p(1^21))=\bigl[(p+1)p](p-1) \bigr]/\,p^{4}.}$ Finally, the
number of binary cubic forms over $\F_p$ having a triple root in
$\P^1_{\F_p}$ is $(p+1)(p-1)$, yielding ${\textstyle
  \mu(T_p(1^3))=(p+1)(p-1)/p^4}$ as desired.
\end{proof}

We next wish to determine the $p$-adic densities of the sets $\U_p$.
Let $\U_p(\cdot)$ denote the subset of elements $f\in T_p(\cdot)$
such that $R(f)$ is maximal at $p$.
If $f$ is an element of $T_p(111)$, $T_p(12)$, or $T_p(3)$, then
$R(f)$ is clearly maximal at $p$, as its discriminant is coprime to $p$.
Thus $\U_p(111)=T_p(111)$, $\U_p(12)=T_p(12)$, and $\U_p(3)=T_p(3)$.
If a binary cubic form $f$ is in $T_p(1^21)$ or $T_p(1^3)$,
then it can clearly be brought into the form
$f(x,y)=ax^3+bx^2y+cxy^2+dy^3$ with $a\equiv b\equiv 0$ (mod $p$),
namely, by sending the unique multiple root of $f$ in $\P^1_{\F_p}$ to
the point $(1,0)$ via a transformation in $\GL_2(\Z)$.  Of all $f\in
T_p(1^2 1)$ or $T_p(1^3)$ that have been rendered in such a form, a
proportion of $1/p$ actually satisfy the congruence $a\equiv 0$
(mod~$p^2$) of condition~(ii).  Thus a proportion of $(p-1)/p$ of
forms in $T_p(1^21)$ and in $T_p(1^3)$ correspond to cubic rings
maximal at~$p$.  We have  proven:

\begin{lemma}\label{udensities}
We have
\[
\begin{array}{rcl}
\mu(\U_p(111))&=&\!\!\frac{1}{6}\,(p-1)^2\;p\;(p+1)\,/\,p^{4}\\[.045in]
\mu(\U_p(12))&=& \!\!\frac{1}{2}\,(p-1)^2\;p\;(p+1)\,/\, p^{4} \\[.045in]
\mu(\U_p(3))&=& \!\!\frac{1}{3}\,(p-1)^2\;p\;(p+1)\,/\, p^{4} \\[.045in]
\mu(\U_p(1^21))&=& \:\;\:\!(p-1)^2\;(p+1)\,/\, p^{4} \\[.045in]
\mu(\U_p(1^3))&=& \:\;\:\!(p-1)^2\;(p+1)\,/\,p^{5}\,. \\[.045in]
\end{array}
\]
\end{lemma}

Following~\cite{DH} let $\V_p$ denote the set of elements $f\in
\U_p$ such that $(f,p)\neq (1^3)$.  Then it is clear from the above
arguments that the elements of $\V_p$ correspond to orders in
\'etale cubic algebras over $\Q$ that are maximal at $p$ and in which $p$ does not totally ramify.
The set $\V_p$ plays an important role in understanding the 3-torsion
in the class groups of cubic fields (see Section 8).

Using the fact that $\U_p$ is simply the union of the
$\,\U_p(\sigma)$'s, while $\mathcal V_p$ is
the union of the $\mathcal U_p(\sigma)$'s where $\sigma\neq
(1^3)$, we obtain from Lemma~\ref{udensities}:

\begin{lemma}\label{uvdensity}
We have
\[
\begin{array}{rcl}
\mu(\U_p)&=&\!\!(p^3-1)(p^2-1)\,/\,p^{5}\\[.045in]
\mu(\V_p)&=& \!\!(p^2-1)^2\,/\,p^{4}\,. \\[.045in]
\end{array}
\]
\end{lemma}



\section{The number of binary cubic forms of bounded discriminant}
Let $V_\R$ denote the vector space of binary cubic forms over $\R$.
Then the action of $\GL_2(\R)$ on $V_\R$ has two nondegenerate orbits,
namely the orbit $V_\R^{(0)}$ consisting of elements having positive
discriminant, and $V_\R^{(1)}$ consisting of those having negative
discriminant.  In this section we wish to understand the number
$N(V_\Z^{(i)};X)$ of \underline{irreducible} $\GL_2(\Z)$-orbits on
$V^{(i)}_\Z := V_\Z\cap V^{(i)}_\R$ having absolute discriminant less than $X$ ($i=0,1$),
where we say that a $\GL_2(\Z)$-orbit on $V_\Z$ is {\it irreducible}
if it consists of binary cubic forms that are irreducible over
$\Q$. In particular, we prove the following strengthening of
Davenport's theorem on the number of $\GL_2(\Z)$-equivalence classes
of irreducible binary cubic forms having bounded discriminant:

\begin{theorem}\label{bcfcount}
$\;N(V^{(0)}_\Z;X)= \displaystyle{\frac{\pi^2}{72}\cdot X + O(X^{5/6})}\,;\;$
$N(V^{(1)}_\Z;X)= \displaystyle{\frac{\pi^2}{24}\cdot X + O(X^{5/6})}\,$.
\end{theorem}
In \cite{Davenport} and \cite{Davenportn}, Davenport had obtained the
main terms of the above theorem with an error bound of $O(X^{15/16})$.

\subsection{Reduction theory}

Define the usual subgroups $K_1,A_+,N$, and $\Lambda$ of $\GL_2(\R)$ as follows:
\begin{eqnarray*}
K_1\,&=&\{\mbox{orthogonal transformations in $\GL_2(\R)$}\};\\
A_+&=&\{a(t):t\in \R_+\},\,\,\mbox{where}\,\,\,a(t)={\footnotesize
\left(\begin{array}{cc} t^{-1} & {} \\ {} & t \end{array}\right)};\\
N\,&=&\{n(u):u\in \R\},\,\,\mbox{where}\,\,\,n(u)={\footnotesize
\left(\begin{array}{cc} 1 & {} \\ u & 1 \end{array}\right)};\\
\Lambda\,&=&\{{\footnotesize
\left(\begin{array}{cc} \lambda & {} \\ {} & \lambda \end{array}\right)}
\}\,\,\mbox{where}\,\,\,\lambda>0.
\end{eqnarray*}

It is well-known (see \cite[Theorem 6.46]{reference}) that the natural
product map $K_1\times A_+\times N\rightarrow \GL_2(\R)$ is an
analytic diffeomorphism. In fact, for any $g\in \GL_2(\R)$, there
exist unique $k\in K_1$, $a=a(t)\in A_+$, $n=n(u)\in N$, and
$\lambda\in\Lambda$ such that $g=k\, a\, n\, \lambda$; this
is the Iwasawa decomposition of $\GL_2(\R)$.

Let $\FF$ denote Gauss's usual fundamental domain for
$\GL_2(\Z)\backslash \GL_2(\R)$ in $\GL_2(\R)$. Then 
$\FF$ may be expressed in the
form $\FF= \{nak\lambda:n\in N'(a),a\in A',k\in
K,\lambda\in\Lambda\}$, where
\begin{equation}
N'(a)= \left\{\left(\begin{array}{cc} 1 & {} \\ {n} & 1 \end{array}\right):
        n\in\nu(a) \right\}    , \;\;
A' = \left\{\left(\begin{array}{cc} t^{-1} & {} \\ {} & t \end{array}\right):
       t\geq \sqrt[4]3/\sqrt{2} \right\}, \;\;
\Lambda = \left\{\left(\begin{array}{cc} \lambda & {} \\ {} & \lambda
        \end{array}\right):
        \lambda>0 \right\},
\end{equation}
and $K$ is as usual the (compact) real special orthogonal group ${\rm
  SO}_2(\R)$; here $\nu(a)$ is the union of either one or two
subintervals of $[-\frac12,\frac12]$ depending only on the value of $a\in
A'$. Furthermore, if $a$ is such that $t\geq 1$, then $\nu(a)=[-\frac12,\frac12]$.
(See, e.g., \cite[Ch.\ 7, Th.\ 1]{Sercia}.)


For $i\in\{0,1\}$, let $n_i$ denote the cardinality of the stabilizer in
$\GL_2(\R)$ of any element $v\in V^{(i)}_\R$ (by the correspondence
of Theorems~\ref{df} and \ref{aut} 
over $\R$, we have $n_1=\Aut_\R(\R^3)=6$ and
$n_2=\Aut_\R(\R\oplus\C)=2$).  Then for any $v\in V_\R^{(i)}$, $\FF v$
will be the union of $n_i$ fundamental domains for the action of
$\GL_2(\Z)$ on $V^{(i)}_\R$.  Since this union is not necessarily
disjoint, $\FF v$ is best viewed as a multiset, where the multiplicity
of a point $x$ in $\FF v$ is given by the cardinality of the set
$\{g\in\FF\,\,|\,\,gv=x\}$. Evidently, this multiplicity is a number
between 1 and $n_i$.

Even though the multiset $\FF v$ is the union of $n_i$ fundamental
domains for the action of $\GL_2(\Z)$ on $V^{(i)}_\R$, not all
elements in $\GL_2(\Z)\backslash V_\Z$ will be represented in $\FF v$
exactly $n_i$ times.  In general, the number of times the
$\GL_2(\Z)$-equivalence class of an element $x\in V_\Z$ will occur in
the multiset $\FF v$ is given by $n_i/m(x)$, where $m(x)$ denotes the
size of the stabilizer of $x$ in $\GL_2(\Z)$. Now the stabilizer in
$\GL_2(\Z)$ of an irreducible element $x\in V_\Z$ is the
group of ring automorphisms of the order corresponding to $x$ under
the Delone--Faddeev correspondence (see Section 2), and is thus either trivial or
$C_3$.  We conclude that, for any $v\in V_\R^{(i)}$, the product
$n_i\cdot N(V_\Z^{(i)};X)$ is exactly equal to the number of
irreducible integer points in $\FF v$ having absolute discriminant
less than $X$, with the slight caveat that the (relatively rare---see
Lemma~\ref{reducible2}) $C_3$-points are to be counted with weight
$1/3$

Now the number of such integer points can be difficult to count
in a single such fundamental domain.  The main technical
obstacle is that the fundamental region $\FF v$ is not bounded, but
rather has a cusp going off to infinity which in fact contains
infinitely many integer points, including many irreducible points.  We
simplify the counting of such points by ``thickening'' the cusp; more
precisely, we compute the number of points in the fundamental region
$\FF v$ by averaging over lots of such fundamental domains, i.e., by
averaging over points $v$ lying in a certain compact subset
$B$ of $V_\R$.

\subsection{Estimates on reducibility}

We first consider the reducible elements in the multiset
$$\RR_X(v):=\{w\in \FF v:|\Disc(w)|<X\},$$ where $v$ is any vector in a
fixed compact subset $B$ of $V_\R$.  Note that if a binary cubic form
$ax^3+bx^2y+cxy^2+dy^3$ satisfies $a=0$, then it is reducible over
$\Q$, since $y$ is a factor. The following lemma, proved in
\cite[Lem.~3]{Davenport} and \cite[Lem.~2]{Davenportn}, shows that
for binary cubic forms in $\RR_X(v)$, reducibility with $a\neq0$ does
not occur very~often.

\begin{lemma}\label{reducible}
  Let $v\in B$ be any point of nonzero discriminant, where $B$ is any
  fixed compact subset of $V_\R$ containing only elements having
  discriminant greater than $1$. Then the number of integral binary
  cubic forms $ax^3+bx^2y+cxy^2+dy^3\in \RR_X(v)$ that are reducible
  with $a\neq 0$ is $O(X^{3/4+\epsilon})$, where the implied constant
  depends only on $B$.
\end{lemma}

\begin{proof}
  For an element $f(x,y)=ax^3+bx^2y+cxy^2+dy^3\in \RR_X(v)$, we have
  $f\in N'A'K\Lambda v$ where $0<\lambda<X^{1/4}$, since
  $\Disc(\lambda\cdot v)=\lambda^4\Disc(v)$. It follows that
  $a=O(\lambda/t^3)=O(X^{1/4})$, $ab=O(\lambda^2/t^4)=O(X^{1/2})$,
  $ac=O(\lambda^2/t^2)=O(X^{1/2})$, $ad=O(\lambda^2)=O(X^{1/2})$, $abc=O(\lambda^3/t^3)=O(X^{3/4})$, and
  $abd=O(\lambda^3/t)=O(X^{3/4})$.  In particular, the latter estimates clearly imply
  that the total number of forms $f\in \RR_X(v)$ with $a\neq 0$ and
  $d=0$ is $O(X^{3/4+\epsilon})$.

  Let us now assume $a\neq 0$ and $d\neq 0$.  Then the above estimates show
  that the total number of possibilities for the triple $(a,b,d)$ is
  $O(X^{3/4+\epsilon})$. Suppose the values $a,b,d$ ($d\neq 0$) are
  now fixed, and consider the possible number of values of $c$ such
  that the resulting form $f(x,y)$ is reducible.  For $f(x,y)$ to be
  reducible, it must have some linear factor $rx+sy$, where $r,s\in\Z$
  are relatively prime.  Then $r$ must be a factor of $a$, while $s$
  must be a factor of $d$; they are thus both determined up to
  $O(X^\epsilon)$ possibilities.  Once $r$ and $s$ are determined,
  computing $f(-s,r)$ and setting it equal to zero then uniquely
  determines $c$ (if it is an integer at all) in terms of $a,b,d,r,s$.
  Thus the total number of reducible forms $f\in \RR_X(v)$ with $a\neq
  0$ is
  $O(X^{3/4+\epsilon})$, as desired.
\end{proof}

We shall need the following lemma, which also follows from \cite[Lemma
2]{Davenport}, bounding the number of integral points in $\RR_X(v)$
that have stabilizer $C_3$ in $\GL_2(\Z)$, when $v$ has positive
discriminant. No integral binary cubic form having negative
discriminant has stabilizer $C_3$ in $\GL_2(\Z)$.
\begin{lemma}\label{reducible2}
  Let $v\in V_\R$ be any point of positive discriminant. Then the number
  of points in $V_\Z\cap\RR_X(v)$ having stabilizer $C_3$ in
  $\GL_2(\Z)$ is $O(X^{3/4+\epsilon})$, where the implied constant is
  independent of $v$.
\end{lemma}

\begin{proof}
  The number of integral points in $\RR_X(v)$ having stabilizer $C_3$
  in $\GL_2(\Z)$ is equal to the number of isomorphism classes of
  cubic rings having automorphism group $C_3$ and discriminant less
  than $X$.  This number is thus independent of $v$, and so it
  suffices to prove the lemma for any single $v$.

  We choose $v$ to be the binary cubic form $x^3 -3xy^2$.  The
  reason for this choice is as follows.  Every binary cubic form
  $f(x,y)=ax^3+bx^2y+cxy^2+dy^3$ has a naturally associated binary
  quadratic form, namely, the ``Hessian covariant''
  $H_f(x,y)=(b^2-3ac)x^2+(bc-9ad)xy+(c^2-3bd)y^2$. It is easy to see
  that if a binary cubic form $f$ is acted upon by an element
  $\gamma\in\SL_2(\Z)$, then $H_f$ is also acted upon by the same
  transformation.  Now $H_v(x,y)=9(x^2+y^2)$, and so $\FF H_v$ consists
  of the usual reduced (positive-definite) binary quadratic forms
  $A_1x^2+A_2xy+A_3y^2$, where $|A_2|\leq A_1\leq A_3$.  Thus $\FF v$
  consists of binary cubic forms satisfying $|bc-9ad|\leq
  b^2-3ac\leq c^2-3bd$.

  Now if a binary cubic form $f$ in $\FF v$ has a nontrivial
  stabilizing element $\gamma$ of order 3 in $\SL_2(\Z)$, then
  $\gamma$ will also stabilize its Hessian $H_f$.  But the only
  reduced binary quadratic form, up to multiplication by scalars,
  having a nontrivial stabilizing element of order 3 is $x^2+xy+y^2$.
  Therefore, any such $C_3$-type binary cubic form
  $f(x,y)=ax^3+bx^2y+cxy^2+dy^3$ in $\FF v$ must satisfy
  \[b^2-3ac=bc-9ad=c^2-3bd.\]
  From this we see that, if $a,b,d$ are fixed, then there is at most
  one solution for $c$.  As in the proof of Lemma~\ref{reducible}, the
  total number of possibilities for the triple $(a,b,d)$ in $\FF v$ is
  $O(X^{3/4+\epsilon})$, and the lemma follows.
\end{proof}

\noindent
In fact, by refining the proof of Lemma~\ref{reducible2},
it can be shown that the number
of $C_3$-points in $\RR_X(v)$ of discriminant less than $X$ is
asymptotic to $cX^{1/2}$, where $c=\pi\sqrt{3}/18$; see~\cite{cyclic}.

Thus, as far as Theorem~\ref{bcfcount} is concerned, the $C_3$-points
in $V_\Z$ are negligible in number and are absorbed in the error term.

\subsection{Averaging}

Let $dv$ denote the usual Euclidean measure on $V_\R$ (normalized so
that $V_\Z$ has co-volume 1) and let $dg=t^{-2}dn\,d^\times
t\,dk\,d^\times\lambda$ be the Haar measure of $\GL_2(\R)$ obtained
from its Iwasawa decomposition (see the beginning of Section 5.1),
where $dk$ is normalized to have measure $1$ on $\text{SO}_2(\R)$.  We
start with a proposition implying that $|\Disc(v)|^{-1}dv$ is a
$\GL_2(\R)$-invariant measure on $V_\R$.
\begin{proposition}\label{volumes}
For $i=0$ or $1$, let $f\in C_0(V_\R^{(i)})$ and let $v_i$ be any
element of $V_\R^{(i)}$.  Then
\[ \int_{g\in \GL_2(\R)} f(g\cdot v_i)\,dg
\,=\,\frac1{2\pi}\int_{v\in\GL_2(\R)\cdot v_i}f(v)\,|\Disc(v)|^{-1}dv\,=\,
\frac{n_i}{2\pi}\int_{v\in V_\R^{(i)}}f(v)\,|\Disc(v)|^{-1}dv . \]
\end{proposition}
The first equality in Proposition~\ref{volumes} is simply a Jacobian
calculation for the change of variable for the map which sends
$g\in\GL_2(\R)$ to $v=g\cdot v_i$ in $V_\R$, where the coordinates for
$g$ are $(k,t,n,\lambda)$, while for $v$ they are the usual Euclidean
coordinates $(a,b,c,d)$ with $dv = da\, db\, dc\, dd$. The second
follows from the fact that the multiset $\GL_2(\R)\cdot v_i$ is an $n_i$-fold cover
of the set $V_\R^{(i)}$.

For a constant $C\geq 1$, let $B=B(C)=\{w=(a,b,c,d)\in V_\R : 3a^2+b^2+c^2+3d^2\leq
C,\;|\Disc(w)|\geq 1\}$; then one easily checks that $B$ is $K$-invariant.
Let $V_\Z^\irr$ denote the subset of irreducible points of $V_\Z$.  It
then follows from the discussion in Section 5.1 that
\begin{equation}
N(V_\Z^{(i)};X) = \frac{\int_{v\in B\cap V_\R^{(i)}}
\#\{x\in \FF v\cap V_\Z^\irr: |\Disc(x)|<X\}\;
|\Disc(v)|^{-1} {d}v}
{n_i\cdot \int_{v\in B\cap V_\R^{(i)}} \:|\Disc(v)|^{-1} dv},
\end{equation}
where points $x\in \FF v\cap V_\Z^\irr$ whose stabilizer in
$\GL_2(\Z)$ is $C_3$ are counted with multiplicity $1/3$.
The denominator of the latter expression is, by construction, a finite
absolute constant
greater than zero.  We have chosen
the measure $|\Disc(v)|^{-1}\,dv$ because it is a
$\GL_2(\R)$-invariant measure.

More generally, for any $\GL_2(\Z)$-invariant subset $S\subset
V_\Z^{(i)}$, let $N(S;X)$ denote the number of irreducible
$\GL_2(\Z)$-orbits on $S$ having discriminant less than $X$.  Let $S^\irr$ denote the subset of irreducible points of $S$.  Then $N(S;X)$ can be expressed as
\begin{equation}\label{eqnsx}
N(S;X) = \frac{\int_{v\in B\cap V_\R^{(i)}}
\#\{x\in \FF v\cap S^\irr: |\Disc(x)|<X\}\;
|\Disc(v)|^{-1} dv}
{n_i\cdot \int_{v\in B\cap V_\R^{(i)}} \:|\Disc(v)|^{-1} dv},
\end{equation}
where, as before, points $x\in \FF v\cap S^\irr$ whose stabilizer in $\GL_2(\Z)$ is $C_3$ are counted with multiplicity $1/3$.
We shall use this as a definition of $N(S;X)$ for any $S\subset
V_\Z$, even if $S$ is not $\GL_2(\Z)$-invariant.  Note that for
disjoint $S_1,S_2\subset V_\Z$, we have $N(S_1\cup S_2;X)=N(S_1;X)+N(S_2;X)$.

Fix $v_i\in V_\R^{(i)}$ and maximal subsets
$H^{(i)}\subset\GL_2(\R)$ such that $H^{(i)}\cdot v_i=B\cap
V_\R^{(i)}$. Thus, the multiset $H^{(i)}\cdot v_i$ is an $n_i$-fold cover of $B\cap V_\R^{(i)}$. The numerator of the right hand side of
Equation (\ref{eqnsx}) is equal to
\begin{equation}\label{eqavging}
\sum_{\substack{{x\in
      S^{\irr}}\\[.02in]{|\Disc(x)|<X}}}\int_{v\in B\cap V_\R^{(i)}} \#\{g \in \FF :
x=gv\} |\Disc(v)|^{-1}dv=
\frac{2\pi}{n_i}\!\!\!\!\!\!\!\sum_{\substack{{x\in S^{\irr}}\\[.02in]{|\Disc(x)|<X}}}\int_{h\in H^{(i)}} \#\{g \in \FF :x=ghv_i\}dh,
\end{equation}
where the equality in (\ref{eqavging}) follows from 
Proposition~\ref{volumes}. 
The right hand side of (\ref{eqavging}) is equal to
\begin{equation}
\frac{2\pi}{n_i}\!\!\!\!\!\!\!\sum_{\substack{{x\in S^{\irr}}\\[.02in]{|\Disc(x)|<X}}}\int_{g\in\FF} \#\{h \in H^{(i)} :x=ghv_i\}dg=\frac{2\pi}{n_i}\int_{g\in\FF}\#\{x\in S^\irr\cap gH^{(i)}v_i:|\Disc(x)|<X\}\,dg.
\end{equation}
Therefore, we have
\begin{eqnarray}\label{impeqsec5}
\!\!\!N(S;X)&\!\!=\!\!&  \frac1{M_i}\int_{g\in\FF}\!\!\!
\#\{x\in S^\irr\cap gB\cap V_\R^{(i)}:|\Disc(x)|<X\}\,dg\\[.075in]
&\!\!=\!\!&  \frac1{M_i}\int_{g\in N'(a)A'\Lambda K}\!\!\!\!\!
\#\{x\in S^\irr\cap n 
\bigl(\begin{smallmatrix}t^{-1}& {}\\
  {}& t\end{smallmatrix}\bigr) 
\lambda k B\cap V_\R^{(i)}:|\Disc(x)|<X\}t^{-2}
dn\, d^\times t\,d^\times \lambda\, dk.
\end{eqnarray}
where 
\begin{equation}\label{midef}
M_i = \frac{n_i}{2\pi}\cdot \int_{v\in B\cap V_\R^{(i)}} \:|\Disc(v)|^{-1} dv.
\end{equation}
Let us write $B(n,t,\lambda,X) = n 
\bigl(\begin{smallmatrix}t^{-1}& {}\\
  {}& t\end{smallmatrix}\bigr) 
\lambda B\cap\{v\in V_\R^{(i)}:|\Disc(v)|<X\}$.
As $KB=B$ and $\int_K dk =1 $, we have
\begin{equation}\label{avg}
N(S;X) = \frac1{M_i}\int_{g\in N'(a)A'\Lambda}
\#\{x\in S^\irr\cap B(n,t,\lambda,X)\}t^{-2}
dn\, d^\times t\,d^\times \lambda\,.
\end{equation}

To estimate the number of lattice points in $B(n,t,\lambda,X)$, we
have the following elementary proposition from the
geometry-of-numbers.  The form we state is essentially due to
Davenport~\cite{Davenport1}.
To state the proposition, we require the following simple definitions.
A multiset $\mathcal R\subset\R^n$ is said to be {\it measurable} if
$\mathcal R_k$ is measurable for all $k$, where $\mathcal R_k$ denotes
the set of those points in $\mathcal R$ having a fixed multiplicity
$k$.  Given a measurable multiset $\mathcal R \subset\R^n$, we define
its volume in the natural way, that is, $\Vol(\mathcal R)=\sum_k
k\cdot\Vol(\mathcal R_k)$, where $\Vol(\mathcal R_k)$ denotes the
usual Euclidean volume of $\mathcal R_k$.

\begin{proposition}\label{genbound}
  Let $\mathcal R$ be a bounded, semi-algebraic multiset in $\R^n$
  having maximum multiplicity $m$, and which is defined by at most $k$
  polynomial inequalities each having degree at most $\ell$.  Let $\RR'$
  denote the image of $\RR$ under any $($upper or lower$)$ triangular,
  unipotent transformation of $\R^n$.  Then the number of integer
  lattice points $($counted with multiplicity$)$ contained in the
  region $\mathcal R'$ is
\[\Vol(\mathcal R)+ O(\max\{\Vol(\bar{\mathcal R}),1\}),\]
where $\Vol(\bar{\mathcal R})$ denotes the greatest $d$-dimensional
volume of any projection of $\mathcal R$ onto a coordinate subspace
obtained by equating $n-d$ coordinates to zero, where
$d$ takes all values from
$1$ to $n-1$.  The implied constant in the second summand depends
only on $n$, $m$, $k$, and $\ell$.
\end{proposition}
Although Davenport states the above lemma only for compact
semi-algebraic sets $\mathcal R\subset\R^n$, his proof adapts without
essential change to the more general case of a bounded semi-algebraic
multiset $\mathcal R\subset\R^n$, with the same estimate applying also to
any image $\mathcal R'$ of $\mathcal R$ under a unipotent triangular
transformation.


  We now have the following lemma on the number of lattice
  points in $B(n,t,\lambda,X)$ with $a\neq 0$:

\begin{lemma}\label{bntlx}
The number of lattice points $(a,b,c,d)$ in $B(n,t,\lambda,X)$ with $a\neq0$ is
$$\left\{
\begin{array}{cl}
0 & \mbox{{\em if} $\frac{C\lambda}{t^3}<1$};\\[.1in]
\Vol(B(n,t,\lambda,X)) + O(\max\{{C^3t^3}
{\lambda^3},1\})&\mbox{{\em otherwise.}}
\end{array}\right.$$
\end{lemma}

\begin{proof}
  From our description of $B$, it follows that the $x^3$-coefficient
  of any binary cubic form in $B$ is bounded by $C$. Thus, if ${C {\lambda/
      t^3}<1}$, then $a=0$ is the only possibility for an integral
  binary cubic form $ax^3+bx^2y+cy^2+dy^3$ in $B(n,t,\lambda,X)$.
   If ${C {\lambda/
      t^3}\geq1}$, then $\lambda$ and $t$ are positive numbers bounded
  from below by $(\sqrt[4]{3}/\sqrt2)^3/C$ and $\sqrt[4]{3}/\sqrt2$ respectively. In this
  case, one sees that the projection of $B(n,t,\lambda,X)$ onto $a=0$
  has volume $O(C^3t^3\lambda^3)$, while all other projections are
  also bounded by a constant times this. The lemma now follows from
  Proposition~\ref{genbound}.
\end{proof}

In~(\ref{avg}), observe that the integrand will be nonzero only if $
t^3 \leq C\lambda$ and $\lambda\leq X^{1/4}$, since $B$ consists only
of points having discriminant at least 1.  Thus we may write, up to an
error of $O(X^{3/4+\epsilon})$ due to Lemma~\ref{reducible}, that
\begin{equation}\label{bigint}
N(V_\R^{(i)};X) 
=\frac1{M_i}
\int_{\lambda=(\sqrt[4]{3}/\sqrt2)^3/C}^{X^{1/4}}
\int_{t=\sqrt[4]{3}/\sqrt2}^{C^{1/3}\lambda^{1/3}}\!\! \int_{N'(t)}
(\Vol(B(n,t,\lambda,X)) +
O(\max\{{C^3t^3}{\lambda^3},1\})) t^{-2}
dn \,d^\times t \,d^\times \lambda.
\end{equation}
The integral of the first summand is
\begin{equation}\label{ugh}
\frac{1}{2\pi M_i}\int_{v\in B\cap V_\R^{(i)}}
\!\!\!\Vol(\RR_X(v))|\Disc(v)|^{-1}dv -  \!\frac1{M_i}
\int_{\lambda=(\sqrt[4]{3}/\sqrt2)^3/C}^{X^{1/4}}
\int^\infty_{C^{1/3}\lambda^{1/3}}\!\int_{N'(t)}
\!\!\!\Vol(B(n,t,\lambda,X)t^{-2} dn d^\times t d^\times \lambda.
\end{equation}
Since $\Vol(\RR_X(v))$ does not
depend on the choice of $v\in V_\R^{(i)}$ (by Proposition~\ref{volumes}),
the first term of (\ref{ugh}) is simply $\Vol(\RR_X(v))/n_i$;
meanwhile, the integral of the second term is easily
evaluated to be $O(C^{10/3}X^{5/6}/M_i(C))$, since
$\Vol(B(n,t,\lambda,X))\ll C^4\lambda^4$.
On the other hand, since $C^3t^3\lambda^3\gg1$ one immediately computes the
integral of the second summand in (\ref{bigint}) to be
$O(C^{10/3}X^{5/6}/M_i(C))$.  We thus obtain, for any $v\in V_\R^{(i)}$,~that
\begin{equation}
N(V_\Z^{(i)};X) = \frac1{n_i}\cdot\Vol(\RR_X(v)) + O(C^{10/3}X^{5/6}/M_i(C)).
\end{equation}

To prove Theorem~\ref{bcfcount}, it remains to compute the fundamental
volume $\Vol(\RR_X(v))$ for $v\in V_\R^{(i)}$.

\subsection{Computation of the fundamental volume}\label{volcomp}



Let $\GL_2^{\pm1}(\R)$ denote the subgroup of elements in
$\GL_2(\R)$ having determinant $\pm1$.
It is known~\cite{reference} (or readily computed using Gauss's
explicit fundamental domain for $\SL_2(\Z)\backslash\SL_2(\R)$\,) that
$\Vol(\GL^{\pm1}_2(\Z)\backslash\GL^{\pm1}_2(\R))=\zeta(2)/\pi$, where this volume
is computed with respect to the measure $dh$ obtained from the Iwasawa
decomposition of $\GL^{\pm1}_2(\R)$.
Then
we obtain using
Proposition~\ref{volumes} that
\[\frac1{n_i}\cdot\Vol(\RR_X(v_i)) = \frac{2\pi}{n_i} \int_{0}^{X^{1/4}}\lambda^4d^\times\lambda
\int_{\GL_2(\Z)\backslash\GL_2^{\pm1}(\R)}dh =
\frac{2\pi}{n_i}\cdot \frac{X}{4}\cdot\frac{\zeta(2)}{\pi}
= \frac{\pi^2}{12n_i}X,\]
This proves Theorem~\ref{bcfcount}, and thus the main term of Theorem~\ref{bcfcount1}.
Together with the Delone-Faddeev correspondence, this also proves the main term
of Theorem~\ref{ringwithres}.



\subsection{Congruence conditions}

We may prove a version of Theorem~\ref{bcfcount} for a set in
$V_\Z^{(i)}$ defined by a finite number of congruence conditions.

\begin{theorem}\label{cong}
Suppose $S$ is a subset of $V^{(i)}_\Z$ defined by finitely many
congruence conditions modulo prime powers. Then we have
\begin{equation}\label{ramanujan}
\lim_{X\rightarrow\infty}\frac{N(S\cap V_\Z^{(i)};X)}{X}
  = \frac{\pi^2}{12n_i}
  \prod_{p} \mu_p(S),
\end{equation}
where $\mu_p(S)$ denotes the $p$-adic density of $S$ in $V_\Z$, and
$n_i=6$ or $2$ for $i=0$ or $1$, respectively.
\end{theorem}

To obtain Theorem~\ref{cong}, suppose $S\subset V^{(i)}_\Z$ is defined
by congruence conditions modulo some integer~$m$.  Then $S$ may be
viewed as the intersection of $V_\Z^{(i)}$ with the union $U$ of (say)
$k$ translates $L_1,\ldots,L_k$ of the lattice $m\cdot V_\Z$.  For
each such lattice translate $L_j$, we may use formula (\ref{avg}) and
the discussion following that formula to compute $N(L_j\cap
V_\Z^{(i)};X)$, where each $d$-dimensional volume is scaled by a factor
of $1/m^d$ to reflect the fact that our new lattice has been scaled by
a factor of $m$.  With these scalings, the volumes of the
$d$-dimensional projections of $B(n,t,\lambda,X)$, for $d=3$, $2$, and
$1$ are seen to be at most $O(m^{-3}C^3t^3\lambda^3)$,
$O(m^{-2}C^2t^4\lambda^2)$, and $O(m^{-1}Ct^3\lambda)$, respectively.
Let $a\geq 1$ be the smallest nonzero first coordinate of any point in
$L_j$.  Then, analogous to Lemma~\ref{bntlx}, the number of lattice
points in $B(n,t,\lambda,X)\cap L_j$ with first coordinate nonzero is
\begin{equation}\label{congcount}
\left\{
\begin{array}{cl}
0 & \mbox{{if} $\frac{C\lambda}{t^3}<a$};\\[.1in]
\displaystyle{\frac{\Vol(B(n,t,\lambda,X))}{m^4} +
O\left(\frac{C^3t^3\lambda^3}{m^3}
+\frac{C^2t^4\lambda^2}{m^2}
+\frac{Ct^3\lambda}{m}+1\right)}&\mbox{{otherwise.}}
\end{array}\right.
\end{equation}
Carrying out the integral for $N(L_j;X)$ as in (\ref{bigint}), we
obtain, up to an error of $O(X^{3/4+\epsilon})$ corresponding to the
reducible points in Lemma~\ref{reducible}, that
\begin{equation}\label{sestimate}
N(L_j\cap V_\Z^{(i)};X) = \frac{\Vol(\RR_X(v))}{m^4}
+ O\left(\frac{1}{M_i(C)}
\left[\frac{C^{10/3}X^{5/6}}{a^{1/3}m^3}+\frac{C^{8/3}X^{2/3}}{a^{2/3}m^2}+
\frac{C^{4/3}X^{1/3}}{a^{1/3}m}+\log\,X\right]\right).
\end{equation}
Assuming $m=O(X^{1/6})$, this gives (up to the $O(X^{3/4+\epsilon})$
reducible points of Lemma~\ref{reducible}):
\begin{equation}\label{sestimate2}
N(L_j;X) = m^{-4}{\Vol(\RR_X(v))} + O(m^{-3}X^{5/6}),
\end{equation}
where the implied constant is again independent of $m$.  Summing
over $j$, we thus obtain
\begin{equation}\label{sestimate3}
N(S;X) = km^{-4}{\Vol(\RR_X(v))} + O(km^{-3}X^{5/6}) + O(X^{3/4}).
\end{equation}
Finally, the identities $km^{-4}=\prod_p\mu_p(S)$
and $\Vol(\RR_X(v))=\pi^2/(12n_i)\cdot X$ yield (\ref{ramanujan}).

Note that (\ref{sestimate})--(\ref{sestimate3}) also give
some information on the {rate} of
convergence of (\ref{ramanujan}) for various $S$, which will indeed be of use
when studying second order terms.

\section{Slicing and second order terms}

In Section 5, we proved that $N(V_\Z^{(i)};X)=c_1^{(i)}X+O(X^{5/6})$,
where $c_1^{(0)}=\pi^2/72$ and $c_1^{(1)}=\pi^2/24$.  Let
$c_2^{(0)}=\sqrt{3}r/30$ and $c_2^{(1)}=r/10$ where
$r=\displaystyle\frac{\zeta(2/3)\Gamma(1/3)(2\pi)^{1/3}}{\Gamma(2/3)}$.
In this section, we prove that
$$N(V_\Z^{(i)};X)=c_1^{(i)}X+c_2^{(i)}X^{5/6}+O(X^{3/4}),$$
thereby proving Theorems \ref{bcfcount1} and \ref{ringwithres}.

\subsection{Proofs of Theorems \ref{bcfcount1} and \ref{ringwithres}}

In Equation (\ref{impeqsec5}) of the previous section (with
$S=V_\Z^{(i)}$), we obtained a formula for the number
$N(V_\Z^{(i)};X)$ in terms of an integral over a chosen fundamental
domain $\FF$ for the left action of $\GL_2(\Z)$ on $\GL_2(\R)$.
Evaluating this integral required us to evaluate the number of
integral points in $B(n,t,\lambda,X)$ for various $n$,~$t$,~$\lambda$,~$X$.
Using Proposition \ref{genbound}, we concluded that the number of
integral points in $B(n,t,\lambda,X)$ is equal to the volume of
$B(n,t,\lambda,X)$ with an error of $O(t^3\lambda^3)$.

In this section, we count points in dyadic ranges of the discriminant. Let
$B(n,t,\lambda,X/2,X)$ be the subset of $B(n,t,\lambda,X)$ that
contains points having discriminant greater than $X/2$ in absolute value. We again
estimate the number of integer points in $B(n,t,\lambda,X/2,X)$ to be
equal to its volume, again with an error of $O(t^3\lambda^3)$.
To obtain a more precise count for the number of lattice points in
$B(n,t,\lambda,X/2,X)$ when $t$ is large, we {\it slice} the set
$B(n,t,\lambda,X/2,X)$ by the coefficient of $x^3$. More precisely,
for $a\in\Z$, let $B_a(n,t,\lambda,X/2,X)$ denote the set of binary
cubic forms in $B(n,t,\lambda,X/2,X)$ whose $x^3$-coefficient is equal
to $a$. Then we have:
\begin{equation}\label{eqintro}
  \#\{x\in V_\Z^{\irr}\cap B(n,t,\lambda,X/2,X)\}=\displaystyle\sum_{\substack{a\in\Z\\a\neq 0}}\#\{x\in V_\Z^{\irr}\cap B_a(n,t,\lambda,X/2,X)\}.
\end{equation}
We then again use Proposition \ref{genbound} to estimate the right
hand side of (\ref{eqintro}).  We shall slice the set
$B(n,t,\lambda,X/2,X)$ when $t$ is ``large''. We separate the large $t$
from the small as follows.

Let $\Psi$ be a smooth function on $\R_{\geq 0}$ such that $\Psi(x)=1$
for $x\leq 2$ and $\Psi(x)=0$ for $x\geq 3$.  Let $\Psi_0$ denote the
function $1-\Psi$. Let $N(V_\Z^{(i)};X/2,X)$ denote the number of
$\GL_2(\Z)$-orbits on $V^{(i),\irr}_\Z$ having discriminant between
$X/2$ and $X$ in absolute value .  Then for any $\kappa>0$, we have just as in
(\ref{avg}) that
\begin{equation}\label{avgeq1}
\begin{array}{rcl}
  N(V_\Z^{(i)};X/2,X) \!\!\!\!\!\!\!\!&\!\!\!\!=\!\!\!\!& \displaystyle{\frac1{M_i}\int_{N'(a)A'\Lambda}
    \Psi\left(\frac{t\kappa}{\lambda^{1/3}}\right)\#\{x\in V_\Z^{(i),\irr}\cap B(n,t,\lambda,X/2,X)\}t^{-2}
    dn\, d^\times t\,d^\times \lambda\,}\vspace{.1in}\\
  &\;\;\;\;\;\;\;\;\;\;\;+\!\! & \displaystyle{\frac1{M_i}\int_{N'(a)A'\Lambda}
    \!\!\Psi_0\left(\frac{t\kappa}{\lambda^{1/3}}\right)\#\{x\in V_\Z^{(i),\irr}\cap B(n,t,\lambda,X/2,X)\}t^{-2}
    dn\, d^\times t\,d^\times \lambda\,.}
\end{array}
\end{equation}
Note that the first summand of the right hand side of (\ref{avgeq1})
is non-zero only when $t<3\lambda^{1/3}/\kappa$, while the second
summand is non-zero only when $t>2\lambda^{1/3}/\kappa$. We will
choose $\kappa$ later to minimize our error term. For now, we merely insist $\displaystyle\lim_{X\rightarrow\infty}\kappa=\infty$ and $\kappa<X^{3/4}$.

Let $D_0$ be a constant that bounds the discriminant of every point in $B$. Since the absolute value of the discriminant of every point in $B$ is
bounded below by $1$ and above by $D_0$, we see that $B(n,t,\lambda,X/2,X)$ is empty unless
$(\frac{X}{D_0})^{1/4}<\lambda<X^{1/4}$. Also, note that $\Psi\left(\frac{t\kappa}{\lambda^{1/3}}\right)$ vanishes whenever $\lambda<27t^3\kappa^3$.
 Thus, by Proposition \ref{genbound}, we see that
the first summand of the right hand side of (\ref{avgeq1}) is
\begin{eqnarray}\label{bigint1}
\frac1{M_i}
\int_{\lambda=(\frac{X}{D_0})^{1/4}}^{X^{1/4}}
\int_{t=\sqrt[4]3/\sqrt{2}}^{3\lambda^{1/3}/\kappa}\!\! \int_{N'(t)}
\!\!\Psi\left(\frac{t\kappa}{\lambda^{1/3}}\right)
(\Vol(B(n,t,\lambda,X/2,X)) +
O(\max\{{t^3}{\lambda^3},1\})) t^{-2}
dn\, d^\times t\, d^\times \lambda.
\end{eqnarray}
The integral of the error term in (\ref{bigint1}) is easily seen to be
$$O\left(\int_{(\frac{X}{D_0})^{1/4}}^{X^{1/4}}
  \int_{t=\sqrt[4]3/\sqrt2}^{\lambda^{1/3}/\kappa}\lambda^3t\;d^{\times}t\;d^\times\lambda\right)=O\left(\frac{X^{5/6}}{\kappa}\right).$$
Therefore, the first summand of the right hand side of (\ref{avgeq1})
is equal to
\begin{equation}\label{eqfirstpart}
\frac{1}{M_i}\int_{\lambda=(\frac{X}{D_0})^{1/4}}^{X^{1/4}}
\int_{t=\sqrt[4]3/\sqrt{2} }^\infty\int_{N'(t)}
\Psi\left(\frac{t\kappa}{\lambda^{1/3}}\right)
\lambda^4\Vol(B(X/(2\lambda^4),X/\lambda^4))t^{-2}dn\,d^\times t\,d^\times\lambda +O\left(\frac{X^{5/6}}{\kappa}\right),
\end{equation}
where $B(d_1,d_2)$ denotes the set of all points in $B$ with discriminant
between $d_1$ and $d_2$.

To evaluate the second summand on the right hand side of
(\ref{avgeq1}), we break up the integrand into a sum over points with
fixed $x^3$-coefficient.  Indeed, we see that it is equal to
\begin{equation}\label{eq211}
  \frac{1}{M_i}\sum_{\substack{a\in\Z\\a\neq 0}}\int_{g\in \FF}\Psi_0\left(\frac{t\kappa}{\lambda^{1/3}}\right)\#\{x\in V_\Z^{(i),\irr}\cap B_a(n,t,\lambda,X/2,X)\}dg.\end{equation}
Since $B$ is $K$-invariant, the number of points in $B_a(n,t,\lambda,X/2,X)$ is equal to the number of points in
$B_{-a}(n,t,\lambda,X/2,X)$. Note that the integrand vanishes for
$a>O(\kappa^3) $ where the implied constant depends only on $B$. We again use Proposition \ref{genbound} to see that (\ref{eq211}) is equal to
\begin{equation}\label{eq222}
  \frac{2}{M_i}\sum_{a=1}^{O(\kappa^3)}\int_{\lambda=(\frac{X}{D_0})^{1/4}}^{X^{1/4}}
  \int_{t=\sqrt[4]3/\sqrt{2}}^{\infty}\! \int_{N'(t)}\Psi_0\left(\frac{t\kappa}{\lambda^{1/3}}\right)(\Vol(B_a(n,t,\lambda,X/2,X))+O(\max\{\lambda^2t^4,1\}))t^{-2}dnd^\times td^\times\lambda.
\end{equation}
Again, we can estimate the integral of the error in (\ref{eq222}) to
be on the order of
\begin{equation}\label{ila}
\sum_{a=1}^{O(\kappa^3)}\int_{\lambda=(\frac{X}{D_0})^{1/4}}^{X^{1/4}}
\int_{t=\sqrt[4]3/\sqrt{2}}^{\lambda^{1/3}/a^{1/3}}\lambda^2t^4\;t^{-2}d^{\times}t\;d^\times\lambda
=X^{2/3}\sum_{a=1}^{O(\kappa^3)}O(a^{-2/3})=O\left(\kappa
  X^{2/3}\right).
  \end{equation}
  We assume from now on that $\kappa\leq \frac13X^{1/12}$.
For sufficiently large values of $X$, it follows that if $\Psi_0(t\kappa/\lambda^{1/3})$ is nonzero, then
$t>\frac{2\lambda^{1/3}}{\kappa}>1$ since $\lambda>(\frac{X}{D_0})^{1/4}$. Thus, the integral over $N'$ in (\ref{eq222}) always goes
between $-1/2$ and $1/2$.  The integral of the main term in
(\ref{eq222}) is now computed to be
\begin{equation}\label{eqsmtcompsitm}
\begin{array}{rcl}
& &\displaystyle{\frac{2}{M_i}\sum_{a=1}^\infty\int_{\lambda=(\frac{X}{D_0})^{1/4}}^{X^{1/4}}
\int_{t>0}\!\!\Psi_0\left(\frac{t\kappa}{\lambda^{1/3}}\right)(\Vol(B_a(0,t,\lambda,X/2,X))t^{-2}d^\times td^\times\lambda}\vspace{.1in}\\
&=&\displaystyle{\frac{2}{M_i}\sum_{a=1}^\infty\int_{\lambda=(\frac{X}{D_0})^{1/4}}^{X^{1/4}}
\int_{t>0}\!\!\Psi_0\left(\frac{t\kappa}{\lambda^{1/3}}\right)\lambda^3t^3\Vol(B_{\frac{at^3}{\lambda}}(X/(2\lambda^4),X/\lambda^4))t^{-2}d^\times td^\times\lambda,}
\end{array}
\end{equation}
where $B_{a}(d_1,d_2)$ denotes the set of forms in $B$ having $x^3$-coordinate equal to $a$ and discriminant between $d_1$ and $d_2$ in absolute value.
We change variables to compute the right hand side of
(\ref{eqsmtcompsitm}); let $u={t^3a}/{\lambda}$ so that $d^\times
u=3d^\times t$. The main term in (\ref{eq222}) is therefore equal to
\begin{equation}\label{eqssbs}
\frac{2}{3M_i}\sum_{a=1}^\infty\int_{\lambda=(\frac{X}{D_0})^{1/4}}^{X^{1/4}}\int_{u>0}
\Psi_0\left(\frac{u^{1/3}\kappa}{a^{1/3}}\right)\frac{\lambda^{10/3}u^{1/3}}{a^{1/3}}
\Vol(B_{u}(X/(2\lambda^4),X/\lambda^4))d^\times ud^\times\lambda.
\end{equation}
To compute the expression above, we first sum over $a$. Let $\Phi(z)$
be equal to $\Psi_0(u^{1/3}/z^{1/3})$.  For a function $F$ defined on the positive reals,
let $\widetilde{F}(s)$ denote the Mellin transforms of $F$.
Since the first derivative $\Psi'_0$ is smooth and Schwartz class, the Mellin
transform $\widetilde{\Psi'_0}(s)$ is holomorphic, entire, and
rapidly decaying on any vertical line $\sigma+it$ as $|t|\to\infty$. Moreover,
by standard properties of the Mellin transform, we have the equality $\widetilde{\Psi'_0}(s+1) = s\widetilde{\Psi_0}(s)$. Thus the functions
$\widetilde{\Psi_0}(s)$ and $\widetilde{\Phi}(s)$ are entire except for a possible simple pole at 0 and rapidly decreasing on vertical lines.
Moreover, the residue at 0 of $\widetilde{\Psi_0}(s)$ is equal to $$\widetilde{\Psi'_0}(1)=\int_{0}^{\infty}\Psi'_0(y)dy = 1.$$
Therefore,
\begin{equation}\label{smtimp}
  \begin{array}{rcl}
\displaystyle\sum_{a=1}^\infty a^{-\frac{1}{3}}\Psi_0\left(\frac{u^{1/3}\kappa}{a^{1/3}}\right)&=&\displaystyle\int_{{\rm Re}\;s=2}\zeta\left(s+1/{3}\right)\widetilde{\Phi}(s)\kappa^{3s}ds\\[.15in]
&=&3\displaystyle\int_{{\rm Re}\;s=2}\zeta\left(s+{1}/{3}\right)\widetilde{\Psi_0}(-3s)(\kappa^3u)^{s}ds\\[.165in]
&=&\zeta\left({1}/{3}\right)+3\widetilde{\Psi_0}(-2)(\kappa^3u)^{2/3}+O_M(\min\{(\kappa^3u)^{-M},1\})  
  \end{array}
\end{equation}
for any integer $M$, where we obtain the last equality by moving the
line of integration to ${\rm Re}\;s=-M$ and computing the residues at $s=0$ and $s=\frac{2}{3}$.  Therefore, (\ref{eqssbs}) is equal to
\begin{eqnarray}\label{eq2sum}
\frac{2}{3M_i}\int_{\lambda=(\frac{X}{D_0})^{1/4}}^{X^{1/4}}\int_{u>0}
\left[\zeta(1/3)+3\widetilde{\Psi_0}(-2)(\kappa^3u)^{2/3}\right]\lambda^{10/3}u^{1/3}
\Vol(B_{u}(X/(2\lambda^4),X/\lambda^4))d^\times ud^\times\lambda,
\end{eqnarray}
with an error of
\begin{equation}\label{eqmissingerror}
O\left(\int_{\lambda=(\frac{X}{D_0})^{1/4}}^{X^{1/4}}\int_{u>0}\min\{(\kappa^3u)^{-1},1\}\lambda^{10/3}u^{1/3}
\Vol(B_{u}(X/(2\lambda^4),X/\lambda^4))d^\times ud^\times\lambda,\right).
\end{equation}
We shall eventually choose $\kappa$ to be equal to $\frac13X^{1/12}$.
Therefore, (\ref{eqmissingerror}) can be bounded above by
\begin{equation}\label{miserrorest}
O\left(\int_{\lambda=(\frac{X}{D_0})^{1/4}}^{X^{1/4}}\int_{u=0}^{\kappa^{-3}}\lambda^{10/3}u^{1/3}d^\times ud^\times\lambda\right)=O\left(\frac{X^{5/6}}{\kappa}\right).
\end{equation}
We now evaluate the integral of the two summands in the integrand of (\ref{eq2sum}) separately. Evaluating
the integral of the second summand, we obtain
\begin{eqnarray*}
&\phantom{=}&\frac{2}{M_i}\int_{\lambda=(\frac{X}{D_0})^{1/4}}^{X^{1/4}}\int_{u>0}\widetilde{\Psi_0}(-2)\kappa^2\lambda^{10/3}u
\Vol(B_{u}(X/(2\lambda^4),X/\lambda^4))d^\times ud^\times\lambda\\
&=&\frac{1}{M_i}\int_{\lambda=(\frac{X}{D_0})^{1/4}}^{X^{1/4}}
\widetilde{\Psi_0}(-2)\kappa^2\lambda^{10/3}\Vol(B(X/(2\lambda^4),X/\lambda^4))d^\times\lambda,
\end{eqnarray*}
which is simply equal to
\begin{equation}\label{eqmt2}
\frac{1}{M_i}\int_{\lambda=(\frac{X}{D_0})^{1/4}}^{X^{1/4}}
\int_{t=0}^\infty
\Psi_0\left(\frac{t\kappa}{\lambda^{1/3}}\right)
\lambda^{\frac{10}3+\frac23}\Vol(B(X/(2\lambda^4),X/\lambda^4))t^{-2}d^\times td^\times\lambda.
\end{equation}
Adding (\ref{eqmt2}) to the main term of (\ref{eqfirstpart}) gives us
the following.
\begin{eqnarray*}
&\phantom{=}&
\frac{1}{M_i}\int_{\lambda=(\frac{X}{D_0})^{1/4}}^{X^{1/4}}
\int_{t=\sqrt[4]3/\sqrt{2}}^\infty\int_{N'(t)}
\left(\Psi\left(\frac{t\kappa}{\lambda^{1/3}}\right)+\Psi_0\left(\frac{t\kappa}{\lambda^{1/3}}\right)\right)
\lambda^4\Vol(B(X/(2\lambda^4),X/\lambda^4))t^{-2}dnd^\times td^\times\lambda\\
&=&\frac{1}{M_i}\int_{\lambda=(\frac{X}{D_0})^{1/4}}^{X^{1/4}}
\int_{t=\sqrt[4]3/\sqrt{2}}^\infty\int_{N'(t)}
(\Vol(B(n,t,\lambda,X/2,X)))t^{-2}dnd^\times td^\times\lambda,
\end{eqnarray*}
which can be evaluated, as in Section 5, to be equal to
$c_1^{(i)}X/2$.

Now the integral of the first summand in (\ref{eq2sum}) is
\begin{equation}\label{eqsmt1}
\frac{2}{3M_i}\int_{\lambda=(\frac{X}{D_0})^{1/4}}^{X^{1/4}}\int_{u>0}
\zeta(1/3)\lambda^{10/3}u^{1/3}\Vol(B_{u}(X/(2\lambda^4),X/\lambda^4))d^\times ud^\times\lambda.
\end{equation}
Let $a(v)$, $b(v)$, $c(v),$ and $d(v)$ denote the four coordinates of points
$v\in B$. Then (\ref{eqsmt1}) is equal to
\begin{eqnarray*}
  &\phantom{=}&\frac{1}{3M_i}\zeta(1/3)\int_{\lambda=(\frac{X}{D_0})^{1/4}}^{X^{1/4}}\int_{B(X/(2\lambda^4),X/\lambda^4)}
  \lambda^{10/3}a(v)^{1/3}\frac{dv}{a(v)}
  d^\times\lambda\\
  &=&\frac{1}{3M_i}\zeta(1/3)\int_{\lambda=(\frac{X}{D_0})^{1/4}}^{X^{1/4}}\int_{B(X/(2\lambda^4),X/\lambda^4)}
  \lambda^{10/3}a(v)^{-2/3}dvd^\times\lambda.
\end{eqnarray*}
Carrying out the integral over $\lambda$, we see that (\ref{eqsmt1}) is equal to
\begin{equation}\label{eqalmostdone}
\frac{1}{10M_i}\zeta(1/3)(1-2^{-5/6})X^{5/6}\int_{B}|\Disc(v)|^{-5/6}a(v)^{-2/3}dv.
\end{equation}
Recalling the definition of $M_i$ in (\ref{midef}), we then see that (\ref{eqsmt1}) is equal
to
$$\frac{2\pi}{10n_i}\zeta(1/3)(1-2^{-5/6})X^{5/6}\displaystyle\frac{\int_{B}|\Disc(v)|^{-5/6}a(v)^{-2/3}dv}{\int_{B}|\Disc(v)|^{-1}dv}.$$

We now evaluate the ratio
\begin{equation}\label{eqratio}\displaystyle\frac{\int_{B}|\Disc(v)|^{-5/6}a(v)^{-2/3}dv}{\int_{B}|\Disc(v)|^{-1}dv}.\end{equation}
The ratio in (\ref{eqratio}) is independent of the $K$-invariant set
$B$. Thus, for any $f\in V_\R^{(i)}$, (\ref{eqratio}) is equal to
$$|\Disc(f)|^{1/6}\int_Ka(\gamma\cdot f)^{-2/3}d\gamma=|\Disc(f)|^{1/6}\int_Kf((1,0)\cdot\gamma)^{-2/3}d\gamma=\frac{|\Disc(f)|}{2\pi}^{1/6}\int_0^{2\pi}f(\cos(\theta),\sin(\theta))^{-2/3}d\theta.$$
We now choose convenient points $f\in V_\R^{(i)}$ for $i=0,1$. For
$i=1$ we choose $f(x,y)=x^3+xy^2$ which has discriminant~$-4$.  Then
$$\frac{|\Disc(f)|}{2\pi}^{1/6}\int_0^{2\pi}f(\cos(\theta),\sin(\theta))^{-2/3}d\theta=\frac{2^{1/3}}{2\pi}\int_{0}^{2\pi}\cos(\theta)^{-2/3}d\theta=\frac{2^{4/3}}{\pi}\int_{0}^{\pi/2}\cos(\theta)^{-2/3}d\theta.$$
The substitution $y=\cos(\theta)$ yields
$$\frac{2^{4/3}}{\pi}\int_{0}^{\pi/2}\cos(\theta)^{-2/3}d\theta=\frac{2^{4/3}}{\pi}\int_0^1y^{-2/3}(1-y^2)^{-1/2}dy.$$
The substitution $z=y^2$ then gives
$$\frac{2^{4/3}}{\pi}\int_0^1y^{-2/3}(1-y^2)^{-1/2}dy=\frac{2^{1/3}}{\pi}\int_0^1z^{-5/6}(1-z)^{-1/2}dz=\frac{2^{1/3}\Gamma(1/6)\Gamma(1/2)}{\pi\Gamma(2/3)},$$
where the final equality follows from evaluating the beta function
${\mathrm{B}}(\frac12,\frac16)$.
Using the standard identities
\begin{equation}
\label{eq:gzidentity}
\begin{array}{rcl}
\Gamma(1/6)&=&\displaystyle{2^{5/3}3^{-1/2}\pi^{3/2}/\Gamma(2/3)^2},\\[.05in]
\Gamma(2/3)&=&\displaystyle{3^{-1/2}2\pi/\Gamma(1/3)},\\[.05in]
\zeta(1/3)&=&\displaystyle{(2\pi)^{-2/3}\Gamma(2/3)\zeta(2/3)},
\end{array}
\end{equation}
we finally see that (\ref{eqalmostdone}) is equal
to $(1-2^{-5/6})c_2^{(1)}X^{5/6}$.

Similarly, for $i=0$ we choose the form $f(x,y)=x^3-3xy^2\in
V_\R^{(0)}$. Using the identity
$\cos(3\theta)=\cos^3(\theta)-3\cos(\theta)\sin^2(\theta)$ we see,
exactly as above, that (\ref{eqalmostdone}) is equal to
$(1-2^{-5/6})c_2^{(0)}X^{5/6}$.  Therefore,
$$N(V_\Z^{(i)};X/2,X)=c_1^{(i)}X/2+c_2^{(i)}(1-2^{-5/6})X^{5/6}+O(X^{2/3}\kappa)+O(X^{5/6}/\kappa),$$
and choosing $\kappa$ to be equal to $\frac13{X^{1/12}}{}$ proves
Theorems \ref{bcfcount1} and \ref{ringwithres}.

\subsection{Congruence conditions}

Let $S\subset V_\Z^{(i)}$ be a $\GL_2(\Z)$-invariant set. We define
$N(S;X/2,X)$ to be the number of irreducible $\GL_2(\Z)$-orbits on $S$ having
discriminant between $X/2$ and $X$ in absolute value.
Identically as in (\ref{avgeq1}), we then~have
\begin{equation*}
\begin{array}{rcl}
  N(S;X/2,X) \!\!\!\!\!\!\!\!&\!\!\!\!=\!\!\!\!& \displaystyle{\frac1{M_i}\int_{N'(a)A'\Lambda}
    \Psi\left(\frac{t\kappa}{\lambda^{1/3}}\right)\#\{x\in
    S^\irr\cap B(n,t,\lambda,X/2,X)\}t^{-2}
    dn\, d^\times t\,d^\times \lambda\,}\vspace{.1in}\\
  &\;\;\;\;\;\;\;\;\;\;\;+\!\! &\displaystyle{\frac1{M_i}\int_{N'(a)A'\Lambda}
    \Psi_0\left(\frac{t\kappa}{\lambda^{1/3}}\right)\#\{x\in
    S^\irr\cap B(n,t,\lambda,X/2,X)\}t^{-2}
    dn\, d^\times t\,d^\times \lambda\,.}
\end{array}
\end{equation*}
We use this as a definition of $N(S;X/2,X)$ even when the set
$S\subset V_\Z^{(i)}$ is not $\GL_2(\Z)$-invariant.

Suppose $\mathcal L\subset V_\Z$ is any sublattice
of index $T$ in $V_\Z$ that is defined by congruence conditions modulo~$m$, so that $mV_\Z\subset L$. In what follows, we
compute $N(\mathcal L\cap V^{(i)}_\Z;X/2,X)$ and $N(\mathcal L\cap V^{(i)}_\Z;X)$, 
for $i=0,1$. The
computation is very similar to that of $N(V_\Z^{(i)};X/2,X)$ and $N(V_\Z^{(i)};X)$,
and we highlight the differences that~occur.

We have
\begin{equation}\label{avgeq2}
\begin{array}{rcl}
 \!\! N(\mathcal L\cap V^{(i)}_\Z;X/2,X)\!\!\!\!\!\!\!\!\!\!
&\!\!\!\!\!\!=\!\!\!\!\!\!& \displaystyle{\frac1{M_i}\int_{N'(a)A'\Lambda}
  \!  \Psi\left(\frac{t\kappa}{\lambda^{1/3}}\right)\#\{x\in\mathcal
    L\cap V_\Z^{(i),\irr}\cap B(n,t,\lambda,X/2,X)\}t^{-2}
    dn\, d^\times t\,d^\times \lambda\,}\vspace{.1in}\\
  &\;\;\;\;\;\;\;\;\;\;\;\!+\!\!\!&\displaystyle{\frac1{M_i}\int_{N'(a)A'\Lambda}
  \!  \Psi_0\left(\frac{t\kappa}{\lambda^{1/3}}\right)\#\{x\in\mathcal
    L\cap V_\Z^{(i),\irr}\cap B(n,t,\lambda,X/2,X)\}t^{-2}
    dn\, d^\times t\,d^\times \lambda\,.}
\end{array}
\end{equation}
Analogously to equation \eqref{congcount}, we see that the first summand of the right
hand side of (\ref{avgeq2}) is equal to
\begin{equation*}
\begin{array}{rl}
  &\displaystyle{\frac{1}{TM_i}\int_{\lambda=(\frac{X}{D_0})^{1/4}}^{X^{1/4}}
  \int_{t=\sqrt[4]{3}/\sqrt{2}}^\infty\int_{N'(t)}
  \Psi\left(\frac{t\kappa}{\lambda^{1/3}}\right)
  \lambda^4\Vol(B(X/(2\lambda^4),X/\lambda^4))t^{-2}dnd^\times td^\times\lambda}\\[.175in]
  &\!\!\!\!\!\!\displaystyle{+ \;\frac{m^4}{TM_i}\int_{\lambda=(\frac{X}{D_0})^{1/4}}^{X^{1/4}}\int_{t=\sqrt[4]{3}/\sqrt{2}}^\infty\int_{N'(t)}
  \Psi\left(\frac{t\kappa}{\lambda^{1/3}}\right)
  \cdot O\left(\frac{t^3\lambda^3}{m^3}
+\frac{t^4\lambda^2}{m^2}
+\frac{t^3\lambda}{m}+1\right)t^{-2}dnd^\times td^\times\lambda}\,.
\end{array}
\end{equation*}
We evaluate the second term above to be
\begin{equation}\label{firstparterror}
O\left( \frac{mX^{5/6}}{T\kappa} + \frac{m^2X^{2/3}}{T\kappa^2} + \frac{m^2X^{1/3}}{T\kappa}+\frac{m^4}{T}\right).
\end{equation}
As in~(\ref{eq211}), we see that the second summand of the right
hand side of~(\ref{avgeq2}) is equal to
\begin{equation}\label{eq2112}
  \frac{1}{M_i}\sum_{\substack{a\in\Z\\a\neq
      0}}\int_\FF\Psi_0\left(\frac{t\kappa}{\lambda^{1/3}}\right)\#\{x\in \mathcal L^{\irr}\cap V^{(i)}_\Z\cap B_a(n,t,\lambda,X/2,X)\}dg.\end{equation}

We write $T=T_1 T_2$, where the $x^3$-coefficient of
every element in $\mathcal L$ is a multiple of $T_1$ and the index of
$\mathcal L_a$ in $V_a$ is equal to $T_2$; here $\mathcal L_a$ (resp.\ $V_a$)
denotes the set of all forms in $\mathcal L$ (resp.\ $V_\Z$) whose $x^3$-coefficient is
equal to $a$.  As in~(\ref{eq211})--(\ref{eqssbs}), we estimate~(\ref{eq2112}) to be
\begin{align*}
&\frac{2}{3T_2M_i}\sum_{\substack{a=1\\T_1|a}}^\infty\int_{\lambda=(\frac{X}{D_0})^{1/4}}^{X^{1/4}}\int_{u>0}
\Psi_0\left(\frac{u^{1/3}\kappa}{a^{1/3}}\right)\frac{\lambda^{10/3}u^{1/3}}{a^{1/3}}
\Vol(B_{u}(X/(2\lambda^4),X/\lambda^4))d^\times ud^\times\lambda\\
&\!\!\!\!\!\!\!+\;\sum_{\substack{a=1\\T_1|a}}^{O(\kappa^3)}\int_{\lambda=(\frac{X}{D_0})^{1/4}}^{X^{1/4}}
\int_{t=\lambda^{1/3}/\kappa}^{\lambda^{1/3}/a^{1/3}}\frac{m^3}{T_2}\cdot O\left(\frac{\lambda^2t^4}{m^2}+\frac{\lambda t^2}{m} +1\right)\;t^{-2}d^{\times}t\;d^\times\lambda.
\end{align*}
The error term is easily integrated to give
\begin{equation}
O\left( \frac{m\kappa X^{2/3}}{T} + \frac{m^2X^{1/4}\kappa}{T} +
  \frac{m^3X^{1/4}}{T}\right). 
\end{equation}
Analogously to the computations in (\ref{smtimp}) and (\ref{eq2sum}), we have
\begin{eqnarray*}
  \sum_{\substack{a=1\\T_1|a}}^\infty a^{-\frac{1}{3}}\Psi_0\left(\frac{u^{1/3}\kappa}{a^{1/3}}\right)&=&T_1^{-1/3}\int_{{\rm Re}\;s=2}\zeta(s+{1}/{3})\widetilde{\Phi}(s)(T_1^{-1/3}\kappa)^{3s}ds\\
  &=&3T_1^{-1/3}\int_{{\rm Re}\;s=2}\zeta(s+{1}/{3})\widetilde{\Psi_0}(-3s)((T_1^{-1/3}\kappa)^3u)^{s}ds\\[.1in]
  &=&T_1^{-1/3}\zeta(1/3)+3\widetilde{\Psi_0}(-2)T_1^{-1}(\kappa^3u)^{2/3}+O_M(T_1^{-1/3}\min\{(T_1^{-1}\kappa^3u)^{-M},1\})
\end{eqnarray*}
for any integer $M$.
Identically as in (\ref{miserrorest}), the error coming from the term
$O_M(T_1^{-1/3}\min\{(T_1^{-1}\kappa^3u)^{-M},1\})$ is equal to $O({X^{5/6}}/{(\kappa T_2}))$.
The total error is thus
$$O\left( \frac{m\kappa X^{2/3}}{T} + \frac{m^2X^{1/4}\kappa}{T} + \frac{m^3X^{1/4}}{T}+\frac{mX^{5/6}}{T\kappa} + \frac{m^2X^{2/3}}{T\kappa^2} + \frac{m^2X^{1/3}}{T\kappa}+\frac{m^4}{T}\right). $$

We will only be interested in the range where $m\leq X^{1/4}$. In this range, we optimize the above by taking $\kappa=X^{1/12}$ to get an error of 
$$O\left( \frac{mX^{3/4}}{T} + \frac{m^2X^{1/2}}{T} + \frac{m^3X^{1/4}}{T}\right) = O\left(\frac{mX^{3/4}}{T}\right).
$$
We thus have the following theorem:
\begin{theorem}\label{shincong}
  Let $\mathcal L\subset V_\Z$ be a sublattice of index $T$ in $V_\Z$, containing $mV_\Z$.
  Write $T=T_1 T_2$, where the $x^3$-coefficient of each element
  in $\mathcal L$ is a multiple of $T_1$ and the corresponding index
  of $\mathcal L_a$ in $V_a$ is equal to~$T_2$. Assume further that $m^4\leq X$. Then
\begin{equation}
N(\mathcal L \cap V^{(i)}_\Z;X/2,X)=\displaystyle\frac{c_1^{(i)}}{T}\frac{X}{2} +
\displaystyle(1-2^{-5/6})\frac{c_2^{(i)}}{T_1^{1/3}T_2}X^{5/6}+O\left(\frac{m}{T}X^{3/4}\right).
\end{equation}
\end{theorem}
Summing over dyadic ranges of the discriminant, we also then obtain
\begin{equation}
  N(\mathcal L \cap V^{(i)}_\Z;X)=\displaystyle\frac{c_1^{(i)}}{T}{X} +
  \displaystyle{\frac{c_2^{(i)}}{T_1^{1/3}T_2}X^{5/6}}
+O\left(\frac{m}{T}X^{3/4}\right).
\end{equation}

\vspace{.1in}
\noindent
{\bf Remark 3.}
  Note that our proof shows that the analogue of Theorem~\ref{shincong}
  also holds for translates of the lattice $\mathcal L$, although 
the constant $\frac{c_2^{(i)}}{T_1^{1/3}T_2}$ would get
  replaced with something rather more complicated.  However, the error
  term would remain the same.

\section{$p$-adic densities for the second term}
Let $p$ be a fixed prime and $\sigma$ be the splitting type $(f,p)$ at
$p$ of an integral binary cubic form $f$.  The methods of the previous
section allow us to count the asymptotic number of $\GL_2(\Z)$-orbits
on $\mathcal U_p(\sigma)$ having bounded discriminant.

More precisely, let us define $\mu_1(\sigma,p)$, $\mu_2(\sigma,p)$,
$\mu_1(p)$, and $\mu_2(p)$ so that
\begin{eqnarray*}
N(\U_p(\sigma)\cap
V^{(i)}_\Z;X)&=&\mu_1(\sigma,p)c_1^{(i)}X+\mu_2(\sigma,p)c_2^{(i)}X^{5/6}
+O_\epsilon(X^{3/4+\epsilon}),\\[.0325in]
N(\U_p;X)&=&\mu_1(p)c_1^{(i)}X+\mu_2(p)c_2^{(i)}X^{5/6}+O_\epsilon(X^{3/4+\epsilon}).
\end{eqnarray*}
We similarly define $\mu'_1(p)$ and $\mu'_2(p)$ so that
\begin{eqnarray*}
N(\V_p;X)&=&\mu'_1(p)c_1^{(i)}X+\mu'_2(p)c_2^{(i)}X^{5/6}+O_\epsilon(X^{3/4+\epsilon}).
\end{eqnarray*}
The values of $\mu_1(\sigma,p)$, $\mu_1(p)$ and $\mu'_1(p)$ were computed in
Section 4 to be equal to $\mu(\U_p(\sigma))$, $\mu(\U_p)$, and $\mu(\V_p)$,
respectively.  In this section we compute the values of
$\mu_2(\sigma,p)$, $\mu_2(p)$ and
$\mu'_2(p)$ 
for all splitting types $\sigma$ and all primes $p$.  We will require
these results to prove Theorems~\ref{main1} and \ref{main2}.

From the results of Section \ref{secpden1}, we see that
$\U_p(111)=T_p(111)$, $\U_p(12)=T_p(12)$, and $\U_p(3)=T_p(3)$. For
$\sigma=(111),(12),(3)$, we write $T_p(\sigma)$ as a union of lattices
in the following way.  For $\alpha, \beta, \gamma\in\P^1_{\overline\F_p}$, let $T_p(\alpha,\beta,\gamma)$ be
the set of all elements $f\in V_\Z$ such that the reduction of $f$
modulo $p$ has roots $\alpha, \beta,$ and $\gamma$ in~$\P^1_{\overline\F_p}$. Then
\begin{eqnarray*}
  T_p(111)&=&\bigcup_{\displaystyle\alpha,\beta,\gamma\in \P^1_{\F_p}}(T_p(\alpha,\beta,\gamma)\setminus p\cdot V_\Z),\\
  T_p(12)&=&\bigcup_{\displaystyle\alpha\in \P^1_{\F_p},\beta_1,\beta_2\in\P^1_{\F_{p^2}}\backslash\P^1_{\F_p}}(T_p(\alpha,\beta_1,\beta_2)\setminus p\cdot V_\Z),\\
  T_p(3)&=&\bigcup_{\displaystyle\gamma_1,\gamma_2,\gamma_3\in\P^1_{\F_{p^3}}\backslash\P^1_{\F_p}}(T_p(\gamma_1,\gamma_2,\gamma_3)\setminus p\cdot V_\Z),
\end{eqnarray*}
where $\alpha,\beta,\gamma$ are distinct points in $\P^1_{\F_p}$, while $\beta_1,\beta_2$ are $\F_p$-conjugate points in $\P^1(\F_{p^2})$ and
$\gamma_1,\gamma_2,\gamma_3$ are $\F_p$-conjugate points in~$\P^1(\F_{p^3})$.

Similarly, the set $T_p(1^21)$ (resp.\ $T_p(1^3)$) can be written as
the union over pairs of distinct points $\alpha,\beta\in\P^1_{\F_p}$ (resp.\
points $\alpha\in\P^1_{\F_p}$) of the sets $T_p(\alpha,\alpha,\beta)$ (resp.\
$T_p(\alpha,\alpha,\alpha)$) which consist of elements $f\in V_\Z$ whose
reduction modulo $p$ has a double root at $\alpha$ and a single root
at $\beta$ (resp.\ a triple root at $\alpha$).  Furthermore, the
results of Section \ref{secpden1} imply that elements $f$ in
$T_p(\alpha,\alpha,\beta)$ or $T_p(\alpha,\alpha,\alpha)$ correspond to rings that
are non-maximal at $p$ if and only if $f(\tilde{\alpha})$ is a
multiple of $p^2$, where $\tilde\alpha$ is any element in $\Z$ whose
reduction modulo $p$ is equal to $\alpha$.

We can now compute the values of $\mu_2(\sigma,p)$ from Theorem
\ref{shincong}.  Let $\sigma=(111)$. We apply Theorem~\ref{shincong}
to the lattices $T_p(\alpha,\beta,\gamma)$ and $p\cdot V_\Z$. For the
lattice $T_p([1:0],\beta,\gamma)$ we have $T_1=p$ and $T_2=p^2$ in the
notation of Theorem \ref{shincong}.  Therefore
$$N(T_p([1:0],\beta,\gamma);X)=\displaystyle\frac{c_1^{(i)}}{p^3}{X} +
\displaystyle{\frac{c_2^{(i)}}{p^{7/3}}X^{5/6}+O_{\epsilon}(X^{3/4+\epsilon})}.$$
For the lattice $T_p(\alpha,\beta,\gamma)$, where none of $\alpha,\beta,$ and
$\gamma$ are equal to $[1:0]\in\P^1_{\F_p}$, we have $T_1=1$ and $T_2=p^3$.
Therefore
$$N(T_p(\alpha,\beta,\gamma);X)=\displaystyle\frac{c_1^{(i)}}{p^3}{X} +
\displaystyle\frac{c_2^{(i)}}{p^{3}}X^{5/6}+O_{\epsilon}(X^{3/4+\epsilon}).$$
Finally for the lattice $p\cdot V_\Z$ we have $T_1=p$ and $T_2=p^3$.
Therefore,
$$N(p\cdot V_\Z;X)=\displaystyle\frac{c_1^{(i)}}{p^4}{X} +
\displaystyle\frac{c_2^{(i)}}{p^{10/3}}X^{5/6}+O_{\epsilon}(X^{3/4+\epsilon}).$$
There are ${p\choose2}$ lattices $T_p([1:0],\beta,\gamma)$ and $p\choose
{3}$ lattices $T_p(\alpha,\beta,\gamma)$ where none of $\alpha,\beta,$ and
$\gamma$ are equal to $[1:0]$.  Thus we have
$$\mu_2((111),p)={p\choose2}(p^{-7/3}-p^{-10/3})+{p\choose3}(p^{-3}-p^{-10/3}).$$

Consider now the splitting type $\sigma=(12)$. Following the above
notation, we have $(T_1,T_2)=(p,p^2)$ for the lattice
$T_p([1:0],\beta_1,\beta_2)$ and $(T_1,T_2)=(1,p^3)$ for
$T_p(\alpha,\beta_1,\beta_2)$ when $\alpha\neq [1:0]$. Since we have
$(p^2-p)/2$ choices for the $\F_p$-conjugate points $\beta_1$ and
$\beta_2$, we have
$$\mu_2((12),p)=\displaystyle\frac{p^2-p}2\left(p(p^{-3}-p^{-10/3})+(p^{-7/3}-p^{-10/3})\right).$$

For $\F_p$-conjugate points
$\gamma_1,\gamma_2,\gamma_3\in\P^1(\F_{p^3})$, the lattice
$T_p(\gamma_1,\gamma_2,\gamma_3)$ has $(T_1,T_2)=(1,p^3)$. Since there
are $(p^3-p)/3$ such triples $(\gamma_1,\gamma_2,\gamma_3)$, we see that
$$\mu_2((3),p)=\displaystyle\frac{p^3-p}3(p^{-3}-p^{-10/3}).$$

When $\sigma=(1^21)$, the situation is slightly more complicated. The
lattice $T_p(\alpha,\alpha,\beta)$ has $(T_1,T_2)=(p,p^2)$ when
$\alpha$ or $\beta$ equals $[1:0]$, and has $(T_1,T_2)=(1,p^3)$
otherwise. To account for the fact that an element $f$ in
$T_p(\alpha,\alpha,\beta)$ corresponds to a ring that is maximal at
$p$ if and only if $f(\tilde{\alpha})$ (where $\tilde{\alpha}$ is an
integer whose reduction modulo $p$ is $\alpha$) is not a multiple of
$p^2$, we must multiply the density of each lattice
$T_p(\alpha,\alpha,\beta)$ by $1-p^{-1/3}$ if $\alpha=[1:0]$ and by
$1-p^{-1}$ if $\alpha\neq[1:0]$. Therefore,
$$\mu_2((1^21),p)=\displaystyle p(p^{-7/3}-p^{-10/3})(1-p^{-1/3})+\bigl(p(p^{-7/3}-p^{-10/3})+p(p-1)(p^{-3}-p^{-10/3})\bigr)(1-p^{-1}).$$

Finally, let $\sigma$ equal $(1^3)$. The lattice
$T_p(\alpha,\alpha,\alpha)$ has $(T_1,T_2)=(p,p^2)$ when
$\alpha=[1:0]$ and $(T_1,T_2)=(1,p^3)$ otherwise.
Therefore, as before,
$$\mu_2((1^3),p)=\displaystyle(p^{-7/3}-p^{-10/3})(1-p^{-1/3})+p(p^{-3}-p^{-10/3})(1-p^{-1}).$$

We list the values of $\mu_1(\sigma,p)$ and $\mu_2(\sigma,p)$ in Table \ref{tab1}.

\renewcommand{\arraystretch}{1.5}

\begin{table}[ht]
\centering
\begin{tabular}{|c | c| c|}
\hline
$\sigma$&$\mu_1(\sigma,p)$&$\mu_2(\sigma,p)$\\[2.1pt]
\hline\hline

$(111)$&$\frac{1}{6}\,(p-1)^2\;p\;(p+1)\,/\,p^{4}$&$p^{-3}\left( {p\choose {3}} (1-p^{-1/3})+\frac{p(p-1)}2(p-1)p^{-1/3}\right)$\\[2.1pt] 
\hline
$(12)$&$\frac{1}{2}\,(p-1)^2\;p\;(p+1)\,/\, p^{4}$&$p^{-3}\left( p\bigl(\frac{p^2-p}2\bigr) (1-p^{-1/3})+\frac{p^2-p}2(p-1)p^{-1/3}\right)$\\[2.1pt] 
\hline
$(3)$&$\frac{1}{3}\,(p-1)^2\;p\;(p+1)\,/\, p^{4} $&$p^{-3}\left(\bigl(\frac{p^3-p}3\bigr)(1-p^{-1/3})\right)$\\[2.1pt] 
\hline
$(1^21)$&$(p-1)^2\;(p+1)\,/\, p^{4} $&$p^{-3}\left( p(p-1)\bigr(1-p^{-1}\bigr)+p(p-1)(1-p^{-1/3})p^{-1/3}\right)$\\[2.1pt]
\hline
$(1^3)$&$(p-1)^2\;(p+1)\,/\,p^{5}$&$p^{-3}\left( p(1-p^{-1/3})\bigl(1-p^{-1}\bigr)+(p-1)(1-p^{-1/3})p^{-1/3}\right)$\\[2.21pt]
\hline
\end{tabular}
\caption{Values of $p$-adic densities for splitting types}\label{tab1}
\end{table}

Adding up the values of the $\mu_1(\sigma,p)$ and the
$\mu_2(\sigma,p)$, both over all $\sigma$ and over all $\sigma\neq (1^3)$, we obtain the following lemma.
\begin{lemma}
We have:
\begin{equation}\label{smtdenmax}
\begin{array}{rclccl}
\mu_1(p)&=&\displaystyle{\left(1-\frac1{p^2}\right)\left(1-\frac1{p^{3}}\right),}\qquad &\mu'_1(p)&=&\displaystyle{\left(1-\frac1{p^2}\right)^2,}\vspace{.1in}\\
\mu_2(p)&=&\displaystyle{\left(1-\frac1{p^2}\right)\left(1-\frac1{p^{5/3}}\right),}\qquad &\mu'_2(p)&=&\displaystyle{\left(1-\frac1{p^2}\right)\left(1-\frac{p^{1/3}+1}{p(p+1)}\right).}
\end{array}
\end{equation}
\end{lemma}

\section{Proofs of the main terms of Theorems~\ref{DHth1}--\ref{gensigmafmt}}

In this section, we use the results of Sections 1--5 to complete the
proofs of the main terms of Theorems~1--8.

We have already proven the main term (indeed even the second main
term) of Theorems \ref{bcfcount1} and \ref{ringwithres}, which give
counts for the number of isomorphism classes of integral binary cubic
forms and cubic orders, respectively, having bounded discriminant. In
fact, Theorem \ref{cong} gives the main term for the count of integral binary
cubic forms satisfying any specified finite set of congruence
conditions.

We recall from Section 3, however, that the set of elements in
$V_\Z$ corresponding to maximal orders is defined by infinitely many
congruence conditions. Similarly, we show in Section 8.1 that the
count in Theorem \ref{DHth2} of $3$-torsion elements in class groups
of quadratic fields is equal to the count of integer binary cubic
forms in another set that too is defined by infinitely many
congruence conditions. To prove that (\ref{ramanujan}) still holds for
such sets, we require a uniform estimate on the error term when
only finitely many factors are taken in (\ref{ramanujan}). This
uniformity estimate is proven in Section~8.2.

In Sections~8.3, 8.4, and 8.5, we then carry out a sieve, using this
uniformity estimate, to prove Theorems~\ref{DHth1}, \ref{gensigmafmt},
and \ref{DHth2} which imply the first main terms of Theorems
\ref{main1}, \ref{gensigma1}, and \ref{main2}, respectively.

\subsection{Cubic fields with no totally ramified primes}

To prove Theorem \ref{DHth2}, we consider those cubic fields in which no prime
is totally ramified.  The significance of being ``nowhere totally
ramified'' is as follows.  Given an $S_3$-cubic field $K_3$, let
$K_{6}$ denote its Galois closure.  Let $K_2$ denote the quadratic field
contained in $K_{6}$ (the ``quadratic resolvent field'').  Then one
checks that the Galois cubic extension $K_6/K_2$ is unramified
precisely when the cubic field $K_3$ is nowhere totally ramified.
Conversely, if $K_2$ is a quadratic field, and $K_6$ is any unramified
cubic extension of $K_2$, then as an extension of the base field $\Q$,
the field $K_6$ is Galois with Galois group $S_3$, and any cubic
subfield $K_3$ of $K_6$ is then nowhere totally ramified.

\subsection{A uniformity estimate}

As in Section~4, let us denote by $\mathcal V_p$ the set of all $f\in
V_\Z$ corresponding to cubic rings $R$ that are maximal at $p$ and in
which $p$ is not totally ramified.  Furthermore, let $\mathcal
Z_p=V_\Z-\mathcal V_p$ (thus $\mathcal Z_p$ consists of those binary
cubic forms whose discriminants are not fundamental).  In order to
apply a simple sieve to obtain Theorems \ref{DHth1}, \ref{DHth2}, and \ref{gensigmafmt}, we require the following proposition:

\begin{proposition}\label{errorestimate}
$N(\mathcal Z_p;X) = O(X/p^2)$,
where the implied constant is independent of $p$.
\end{proposition}

\begin{proof}
  The set $\mathcal Z_p$ may be naturally partitioned into two
  subsets: $\mathcal W_p$, the set of forms $f\in V_\Z$
  corresponding to cubic rings not maximal at $p$; and $\mathcal
  Y_p$, the set of forms $f\in V_\Z$ corresponding to cubic
  rings that are maximal at $p$ but also totally ramified at $p$.

  We first treat $\mathcal W_p$.  Recall that the {\it content} $\ct(R)$
  of a cubic ring $R$ is defined as the maximal integer $n$ such that
  $R=\Z+nR'$ for some cubic ring $R'$.  It follows from
  (\ref{ringlaw3}) that the content of $R$ is simply the content
  (i.e., the greatest common divisor of the coefficients) of the
  corresponding binary cubic form $f$.  We say $R$ is {\it primitive}
  if $\ct(R)=1$, and $R$ is {\it primitive at $p$} if $\ct(R)$ is not a
  multiple of $p$.
The following lemma follows immediately from Proposition~\ref{subring}.

\begin{lemma}\label{atmost3}
  Suppose $R$ is a cubic ring that is primitive at $p$.  Then the
  number of subrings of index $p$ in $R$ is at most $3$.
\end{lemma}


To prove the proposition, suppose $R$ is a cubic ring of absolute
discriminant less than $X$ that is not maximal at $p$.  By
Lemma~\ref{nonmax}, the cubic ring $R$ has a $\Z$-basis
$\langle1,\omega,\theta\rangle$ such that either (i)
$R'=\Z+\Z\cdot(\omega/p)+\Z\cdot\theta$ forms a cubic ring, or (ii)
$R''=\Z+\Z\cdot(\omega/p)+\Z\cdot(\theta/p)$ forms a cubic ring.

Assume we are in case (i), i.e., $R'$ is a ring. If $R'$ is primitive
at $p$, then we have that $\Disc(R')=\Disc(R)/p^2<X/p^2$; thus the
total number of possible rings $R'$ that can arise is $O(X/p^2)$ by
Theorem~\ref{ringwithres}.  By Lemma~\ref{atmost3}, the number of $R$
that can correspond to such $R'$ is at most three times that, which is
also $O(X/p^2)$.  On the other hand, if $R'$ is not primitive at $p$,
then let $S$ be the ring such that $R'=\Z+pS$.  Then
$\Disc(S)=\Disc(R)/p^6<X/p^6$, so the number of possibilities for $S$
is $O(X/p^6)$, which is thus the number of possibilities for $R'$
(since $R'=\Z+pS$).  The number of possibilities for $R$ is then $p+1$
(the number of index $p$ submodules of a free $\Z$-module of rank 2) times the
number of possibilities for $R'$, yielding $O((p+1)X/p^6)$
possibilites.  We conclude that in case (i), the number of
possibilities for $R$ is $O(X/p^2)+O((p+1)X/p^6)=O(X/p^2)$.

Assume we are now in case (ii), i.e., $R''$ is a ring.  Then
$R=\Z+pR''$ where $\Disc(R'')=\Disc(R)/p^4<X/p^4$.  The number of
possible $R''$ in this case is $O(X/p^4)$ by
Theorem~\ref{ringwithres}, and so the number of possible cubic rings
$R=\Z+pR''$ arising from case (ii) is $O(X/p^4)$.
Thus the total number $N(\mathcal W_p;X)$ of cubic rings $R$ that are
not maximal at $p$ and have absolute discriminant less than $X$ is
$O(X/p^2)+O(X/p^4)=O(X/p^2)$, as desired.


Finally, that $N(\mathcal Y_p;X)=O(X/p^2)$ follows easily from
class field theory.  A nice, short exposition of this may be found in,
e.g., \cite[p.\ 15]{DW}.
\end{proof}



%
%
%
%

\subsection{Density of discriminants of cubic fields (Proof of Theorem~1)}



We may now prove Theorem~\ref{DHth1}.
Let $\mathcal U=\cap_p \mathcal U_p$.  Then $\mathcal U$
is the set of $v\in V_\Z$ corresponding to maximal cubic rings
$R$.
By Lemma~\ref{uvdensity}, the $p$-adic density of $\U_p$ is given by
$\mu(\mathcal U_p)= (1-p^{-2})(1-p^{-3})$.
Suppose $Y$ is any positive integer.  It follows from
(\ref{ramanujan}) that
\[\lim_{X\rightarrow\infty} \frac{N(\cap_{p<Y}\mathcal U_p\cap V^{(i)}_\Z ;X)}{X}
= \frac{\pi^2}{12n_i}
    \prod_{p<Y}[(1-p^{-2})(1-p^{-3})].
\]
Letting $Y$ tend to $\infty$, we obtain immediately that
\[
\displaystyle{\limsup_{X\rightarrow\infty} \frac{N(\mathcal
    U\cap V^{(i)}_\Z;X)}{X}}
  \leq \displaystyle{\frac{\pi^2}{12n_i}
    \prod_p
    [(1-p^{-2})(1-p^{-3})]}
=\frac{1}{2n_i\zeta(3)}.
\]
To obtain a lower bound for $N(\mathcal U\cap V^{(i)}_\Z;X)$, we note that
\[\bigcap_{p<Y} \mathcal U_p \subset
(\mathcal U \cup \bigcup_{p\geq Y}\mathcal W_p).\]
Hence by Proposition~\ref{errorestimate},
\[
\lim_{X\rightarrow\infty}
\frac{N(\U\cap V^{(i)}_\Z;X)}{X}\geq\frac{\pi^2}{12n_i}
   \prod_{p<Y}[(1-p^{-2})(1-p^{-3})]
      - O(\sum_{p\geq Y} p^{-2}).
\]
Letting $Y$ tend to infinity completes the proof.

We note that the same arguments also apply when counting cubic fields with
specified local behavior at finitely many primes.


\subsection{A simultaneous generalization (Proof of Theorem~\ref{gensigmafmt})}

We now prove Theorem~\ref{gensigmafmt}, which gives the
density of discriminants of cubic orders or fields satisfying any
finite number (or in many natural cases, an infinite number) of local
conditions. To this end, for each prime~$p$ let $\Sigma_p$ be a
set of isomorphism classes of nondegenerate cubic rings over $\Z_p$.
(By {\it nondegenerate}, we mean having nonzero discriminant over
$\Z_p$, so that it can arise as $R\otimes\Z_p$ for some cubic order
$R$ over $\Z$.)  We denote the collection $(\Sigma_p)$ of these local
specifications over all primes $p$ by $\Sigma$.  We say that the
collection $\Sigma=(\Sigma_p)$ is {\it acceptable} if, for all
sufficiently large $p$, the set $\Sigma_p$ contains at least the
{maximal} cubic rings over $\Z_p$ that are not totally ramified at
$p$.

For a cubic order $R$ over $\Z$, we write ``$R\in\Sigma$'' (or say
that ``$R$ is a $\Sigma$-order'') if $R\otimes\Z_p\in\Sigma_p$ for all~$p$. We wish to determine the number of $\Sigma$-orders $R$ of bounded
discriminant, for any acceptable collection $\Sigma$ of local
specifications.

To this end, fix an acceptable $\Sigma=(\Sigma_p)$ of local
specifications, and also fix any $i\in\{0,1\}$.  Let $S=S(\Sigma,i)$
denote the set of all irreducible $f\in V_\Z^{(i)}$ such that the
corresponding cubic ring $R(f)\in\Sigma$.  Then the number of
$\Sigma$-orders with discriminant at most $X$ is given by $N(S;X)$.
We prove the following asymptotics for $N(S;X)$.

\begin{theorem}\label{gensigma2}We have
$\displaystyle{\lim_{X\to\infty}\frac{N(S(\Sigma,i);X)}{X}\,=\,\frac1{2n_i}
\prod_p\Bigl(\frac{p-1}{p}\cdot \sum_{R\in\Sigma_p}
\frac{1}{\Disc_p(R)}\cdot\frac1{|\Aut(R)|}\Bigr).}$
\end{theorem}

Although $S=S(\Sigma,i)$ might again be defined by infinitely many congruence
conditions, the estimate provided in Proposition~\ref{errorestimate}
(and the fact that $\Sigma$ is acceptable) shows that equation
(\ref{ramanujan}) continues to hold for the set $S$; the argument is
identical to that in the proof of Theorem~1.

We now evaluate $\mu_p(S)$ in terms of the cubic rings lying in $\Sigma_p$.

\begin{lemma}\label{ramanujan11}
We have
  $$\displaystyle{\mu_p(S(\Sigma,i)) = \frac{\#\GL_2(\F_p)}{p^4}\cdot
\sum_{R\in\Sigma_p} \frac{1}{\Disc_p(R)}\cdot\frac1{|\Aut(R)|}.}$$
\end{lemma}

\begin{proof}
  The proof of Theorem~\ref{df}, with $\Z_p$ in place of $\Z$, shows that
  for any cubic $\Z_p$-algebra $R$ there is a unique element
  $v\in V_{\Z_p}$ up to $\GL_2(\Z_p)$-equivalence satisfying
  $R_{\Z_p}(v)=R$.
  Moreover, the automorphism group of such a cubic $\Z_p$-algebra $R$
  is simply the size of the stabilizer in $\GL_2(\Z_p)$ of the
  corresponding element $v\in V_{\Z_p}$ (cf.\ Prop.\ \ref{aut}).

  We normalize the Haar measure $dg$ on the $p$-adic group $\GL_2(\Z_p)$ so
  that $\int_{g\in \GL_2(\Z_p)}dg=\#\GL_2(\F_p)$.  Since
  $|\Disc(x)|_p^{-1}\cdot dx$ is a $\GL_2(\Q_p)$-invariant measure on
  $V_{\Z_p}$, we must have for any cubic $\Z_p$-algebra $R=R(v_0)$
that
\[ \int_{{x\in V_{\Z_p}}\atop{R(x)=R}} dx
=c\cdot\int_{g\in \GL_2(\Z_p)/\Stab(v_0)} |\Disc(gv_0)|_p\cdot dg
=c\cdot\frac{|\Disc(R)|_p\cdot \#\GL_2(\F_p)}
{\#\Aut_{\Z_p}(R)},
\]
for some constant $c$.
A Jacobian calculation using an indeterminate $v_0$ satisfying
$\Disc(v_0)\neq 0$ shows that $c=p^{-4}$, independent of $v_0$.  The
lemma follows.
\end{proof}

Finally, we observe that $\#\GL_2(\F_p)=(p^2-1)(p^2-p)$, and so
\[\frac{\pi^2}{12n_i}\prod_p\mu_p(S(\Sigma,i)) =
\frac{\pi^2}{12n_i}\prod_p \Bigl(1-\frac1{p^2}\Bigr)
\Bigl(\frac{p-1}{p}\Bigr)\cdot \sum_{R\in\Sigma_p}
\frac{1}{\Disc_p(R)}\cdot\frac1{|\Aut(R)|},\] proving
Theorem~\ref{gensigma2}.  Noting that $n_1=\Aut_\R(\R^3)$ and
$n_2=\Aut_\R(\R\oplus\C)$ then yields Theorem~\ref{gensigmafmt}.

\vspace{.1in}
\noindent
{\bf Remark 4.}
Lemma~\ref{ramanujan11}, together with the identities
$\mu_p(V_{\Z_p})=1$ and $\mu_p(\mathcal U_p)=(p^3-1)(p^2-1)/p^5$ of
Lemma~\ref{uvdensity}, give the interesting formulae
\begin{equation}\label{ordermass}
 \sum_{R {\rm \; nondeg.\; cubic\; ring } \,/\, \Z_p }
\frac{1}{\Disc_p(R)}\cdot\frac1{|\Aut(R)|}\,=\,
\Bigl(1-\frac1p\Bigr)^{-1}\Bigl(1-\frac1{p^2}\Bigr)^{-1}
\end{equation}
and
\begin{equation}\label{fieldmass}
  \sum_{K {\rm \; etale \;cubic \;extension \;of\;} \Q_p }
  \frac{1}{\Disc_p(K)}\cdot\frac1{|\Aut(K)|}\,=\, 1 + \frac 1p +
  \frac1{p^2}\,.\,\,\,\,\,\,\,\;\;\;\;\;\,\,\,\,\;\;\;\;\,\,\,
\end{equation}
(Note that (\ref{ordermass}) is an infinite sum!)  What is remarkable
about these formulae is that their statements are independent of $p$.
Such ``mass formulae'' for local fields and orders in fact hold in far
more generality (in particular, for degrees other than 3); see
\cite{Serre}, \cite{Bhamass1}, and \cite{Bhamass2}.

\subsection{The mean size of the 3-torsion subgroups of class groups
  of quadratic fields}

In this section we prove Davenport and Heilbronn's theorem on the
average size of the 3-torsion subgroups of class groups of quadratic
fields.  This is accomplished using class field theory, as in
Davenport and Heilbronn's original arguments.  This will prove
Theorem~\ref{DHth2}.



%

Let $\mathcal V=\cap_p \mathcal V_p$ be the set of all $v\in V_\Z$
corresponding to maximal cubic rings that are nowhere totally ramified
(as in Section~4).  Then by Lemma~\ref{uvdensity}, we have
$\mu(\mathcal V_p)= (1-p^{-2})^2.$
By the same argument as in the proof of the main term of Theorem~\ref{main1},
\[
\lim_{X\rightarrow\infty}
\frac{N(\V\cap V^{(i)}_\Z;X)}{X}=\frac{\pi^2}{12n_i}
   \prod [(1-p^{-2})^2]=\frac{3}{n_i\pi^2}.
\]

Now given a nowhere totally ramified cubic field $K_3$, we have
observed earlier that in the Galois closure $K_{6}$ is contained a
quadratic field $K_2$ and $K_6/K_2$ is unramified.  In addition, the
discriminant of $K_2$ is equal to the discriminant of $K_3$.
Furthermore, by class field theory the number of triplets of cubic
fields $K_3$ corresponding to a given $K_2$ in this way equals
$(h_3^\ast(K_2) - 1)/2$, where $h_3^\ast(K_2)$ denotes the number of
3-torsion elements in the class group of $K_2$.  Therefore,
\begin{equation}\label{l4eq}
  \begin{array}{ccc}
    \displaystyle{\sum_{{0<\Disc(K_2)<X}} (h_3^*(K_2)-1)/2} &=& N(\V\cap
    V^{(0)}_\Z;X), \\[.3in]
    \displaystyle{\sum_{{\!\!\!-X<\Disc(K_2)<0}}\!\!
      (h_3^*(K_2)-1)/2} &=& N(\V\cap V^{(1)}_\Z;X).
  \end{array}
\end{equation}
Since it is known that
\begin{equation}\label{trDH}
\begin{array}{ccc}
  \displaystyle{\lim_{X\rightarrow\infty}\frac{\sum_{{0<\Disc(K_2)<X}} 1}{X}}
  &=&  \displaystyle{\frac{3}{\pi^2}}, \\[.25in]
  \displaystyle{\lim_{X\rightarrow\infty}\frac{\sum_{{-X<\Disc(K_2)<0}} 1}{X}}
  &=&  \displaystyle{\frac{3}{\pi^2}},
\end{array}
\end{equation}
we conclude
\[
\begin{array}{ccccccc}
  \displaystyle{
    \lim_{X\rightarrow\infty}\frac{\sum_{{0<\Disc(K_2)<X}} h_3^*(K_2)}
    {\sum_{{0<\Disc(K_2)<X}} 1} }
  &\!\!=\!\!&
  \displaystyle{
    1+2\lim_{X\rightarrow\infty}\frac{N(\V\cap V^{(0)}_\Z;X)}{\sum_{{0<\Disc(K_2)<X}} 1}}
  &\!\!=\!\!&
  \displaystyle{1+\frac{2\cdot 3/6\pi^2}{3/\pi^2}}&\!\!=\!\!&
  \displaystyle{\frac{4}{3}},\\[.35in]
  \displaystyle{
    \lim_{X\rightarrow\infty}\frac{\sum_{{-X<\Disc(K_2)<0}} h_3^*(K_2)}
    {\sum_{{-X<\Disc(K_2)<0}} 1} }
  &\!\!=\!\!&
  \displaystyle{
    1+ 2\lim_{X\rightarrow\infty}\frac{N(\V\cap V^{(1)}_\Z;X)}{\sum_{{-X<\Disc(K_2)<0}} 1}}
  &\!\!=\!\!&
  \displaystyle{1+\frac{2\cdot 3/2\pi^2}{3/\pi^2}} &\!\!=\!\!&
  \displaystyle{2}.
\end{array}
\]

\section{A refined sieve, and proofs of Theorems~\ref{main1}, \ref{main2}, and \ref{gensigma1}}

As we have seen, an integral binary cubic form corresponds to a maximal
ring if and only if its coefficients satisfy certain congruence
conditions modulo $p^2$ for each prime $p$.  To prove Theorem
\ref{main1} using Theorem~\ref{shincong}, we require a suitable sieve
as follows. Recall that for each prime $p$, we defined $\W_p$ to be
the set of binary cubic forms corresponding to cubic rings that are non-maximal at $p$, and $\ZZ_p$
to be the set of binary cubic forms corresponding to cubic rings that are non-maximal at $p$, or are maximal at $p$ but in which $p$ is totally
ramified. For a squarefree integer $n$, define $\W_n=\cap_{p|n}\W_p$
and $\ZZ_n=\cap_{p|n}\ZZ_p$. Then the number of isomorphism classes of
maximal cubic orders having absolute discriminant in the dyadic range $X/2$ to $X$ is
equal to
\begin{equation}\label{mainsieve}
N(\U\cap V_\Z^{(i)};X/2,X) = \displaystyle\sum_{n\in\mathbb{N}}\mu(n)N(\W_n\cap V_\Z^{(i)};X/2,X)
\end{equation}
and the number of isomorphism classes of nowhere totally ramified maximal cubic 
orders in the range $X/2$ to $X$ is equal to
\begin{equation}\label{mainsieve1}
N(\V\cap V_\Z^{(i)};X/2,X) = \displaystyle\sum_{n\in\mathbb{N}}\mu(n)N(\ZZ_n\cap V_\Z^{(i)};X/2,X).
\end{equation}

We focus our discussion on the first sieve, the second sieve being treated in an analogous manner.
In order to prove Theorem \ref{main1}, we need to estimate the
individual terms on the right hand side of (\ref{mainsieve})
accurately. The difficulty lies in the fact that the sets $\W_n$ are
defined by congruence conditions modulo $n^2$. We are then not able to
directly apply Theorem~\ref{shincong}, due to the fact that the
$\W_n$ is the union of a large number of lattices modulo $n^2$. In
Section 9.1, we show how to transform this count to one over fewer
lattices defined by congruence conditions modulo $n$, thus enabling us to
use Theorem~\ref{shincong} more effectively.

We then split (\ref{mainsieve}) into three ranges for $n$ and use a
different method on each range.  We use the splitting of
the discriminant range into dyadic ranges so that we may choose the
three ranges for $n$ depending on the dyadic range of the
discriminant.  When $n$ is small, we use Theorem~\ref{shincong}
together with an identity proven in Section 9.1 to evaluate
$N(\W_n;X/2,X)$ with two main terms and a smaller error term.
Meanwhile, when $n$ gets very large we apply the uniformity estimates
from \cite[Lemma~2.7]{BBP} to bound the size of $|N(\W_n;X/2,X)|$.
Lastly, when $n$ is around $X^{1/6}$ it turns out that
Theorem~\ref{shincong} and \cite[Lemma~2.7]{BBP} do not suffice, and
so we require a different argument. We use again the correspondence of
Section~$9.1$ to reduce the problem to one of determining the main
term for the weighted count of binary cubic forms having bounded
discriminant, where each binary cubic form is weighted by the number
of its roots in $\P^{1}(\Z/n\Z)$.  To accomplish this count, we us an equidistribution argument, carried out in Section~9.4.
 We then complete the proof of Theorem~\ref{main1} in Section~9.5.

In Section 9.6, we prove Theorem \ref{main2} in a very similar manner
to the proof of Theorem \ref{main1}.  Finally, in Section~9.7, we prove
Theorem~\ref{gensigma1} by expressing the second terms that arise in the count of isomorphism
classes of cubic rings of bounded discriminant satisfying specified
local conditions in terms of local masses of cubic rings.


\subsection{Two useful identities}

For $\alpha\in\P^1(\Z/p\Z)$, define $V_{p,\alpha}$ to be the set of
all integer binary cubic forms $f\in V_\Z$ such that $f$ (mod~$p$) has
a root at $\alpha$, and $V^2_{p,\alpha}$ the set of
all integer binary cubic forms $f\in V_\Z$ such that $f$ (mod~$p$) has
at least a double root at $\alpha$.
Note that
although $V_{p,\alpha}$ and $V^2_{p,\alpha}$ are not $\GL_2(\Z)$-invariant, the unions
$\displaystyle\cup_\alpha V_{p,\alpha}$
and $\displaystyle\cup_\alpha V^2_{p,\alpha}$
are each $\GL_2(\Z)$-invariant.

Our sieve makes use of the following proposition which contains
two essential identities:

\begin{proposition}\label{switch}
We have
\begin{eqnarray}\label{switchprimes}
\label{firstswitch}N(\mathcal W_p;X) &=& \displaystyle\sum_{\alpha\in\mathbb{P}^1(\mathbb{F}_p)} N(V_{p,\alpha};X/p^2)
  \;\,-\,\displaystyle\sum_{\alpha\in\mathbb{P}^1(\F_p)}
  N(V_{p,\alpha};X/p^4) \;+\; N(V_\Z;X/p^4) \,; \\[.2in]
\label{secondswitch}N(\ZZ_p;X) &=& \displaystyle\sum_{\alpha\in\mathbb{P}^1(\mathbb{F}_p)} N(V_{p,\alpha};X/p^2)+N(T_p(1^3);X)
  \,-\!\displaystyle\sum_{\alpha\in\mathbb{P}^1(\F_p)}
  N(V^2_{p,\alpha};X/p^2) \;+\; N(V_\Z;X/p^4) \,.
\end{eqnarray}
\end{proposition}

\begin{proof}
  To prove (\ref{firstswitch}), we count isomorphism classes of pairs
  $(R,R')$ of cubic rings such that $R\subset R'$ with $[R':R]=p$ and
  $\Disc(R)<X$.  We count these in two ways, namely, by $R$ and by
  $R'$.

  First, in order to count pairs $(R,R')$ by $R$, recall from
  Proposition~\ref{supring} that, for any integral binary cubic form
  $f\in\W_p\setminus p\cdot V_{\Z}$, the ring $R=R(f)$ is contained in
  a unique ring $R'\subset R\otimes\Q$ such that $[R':R]=p$.  Meanwhile, if
  $f =pg\in p\cdot V_\Z$, then $R$ sits inside $\omega_p(g)$ rings
  $R'\subset R\otimes \Q$ with $[R':R]=p$, where we use $\omega_p(g)$
  to denote the number of roots in $\P^1(\Z/p\Z)$ of $g$ (mod~$p$).  It
  follows that the total number of pairs $(R,R')$ is
\begin{equation}\label{loosepairscount}
  N(\mathcal W_p;X) - N(V_\Z;X/p^4)+ \displaystyle\sum_{\alpha\in\mathbb{P}^1(\F_p)} N(V_{p,\alpha};X/p^4).
\end{equation}
The third term on the right hand
side of the above expression counts those pairs $(R\!=\!R(f),R')$ that
correspond to integer binary cubic forms $f=pg\in p V_\Z$.

We now count the number of pairs $(R,R')$ by $R'$.  Recall by
Proposition~\ref{subring} that for any binary cubic form $f$, the
cubic ring $R'=R(f)$ has precisely $\omega_p(f)$ subrings $R$ of index
$p$.  Therefore, since $R'$ is constrained by $\Disc(R')=\Disc(R)/p^2<X/p^2$, we
see then that the total number of pairs $(R,R')$ is given by
\begin{equation}\label{tightpairscount}
\displaystyle\sum_{\alpha\in\mathbb{P}^1(\F_p)} N(V_{p,\alpha};X/p^2).
\end{equation}
Equating (\ref{loosepairscount}) and (\ref{tightpairscount}) yields
the identity (\ref{firstswitch}).

\vspace{.1in}
To prove (\ref{secondswitch}), we begin by deriving a formula for $N(\mathcal W_p\cap T_p(1^3); X)$.
To this end, we count now isomorphism classes of
pairs $(R,R')$ of cubic rings such that $R\subset R'$ with $[R':R]=p$
and $\Disc(R)<X$, where furthermore $R$ has splitting type $(1^21)$
at $p$.  We again count these in two ways, namely, by $R$ and by~$R'$.

First, we note that if $R$ has splitting type $(1^21)$ at $p$, and
$R=R(f)$, then $R'\subset R\otimes\Q$ is uniquely determined and is
primitive at $p$; moreover, if we write $R'=R(f')$, then $f'$ (mod $p$)
has a distinguished simple root in $\P^1(\F_p)$.  Conversely, if
$R'=R(f)$, where $f$ (mod $p$) has a simple root in $\P^1(\F_p)$, then
any subring $R$ of index $p$ will have splitting type $(1^21)$ at $p$.
It follows that the
number of desired pairs $(R,R')$ is
\begin{equation}\label{pairscount2}
  N(\mathcal W_p\cap T_p(1^21);X) \;\;=\;\:
\displaystyle\sum_{\alpha\in\mathbb{P}^1(\F_p)} N(V_{p,\alpha};X/p^2)
- \displaystyle\sum_{\alpha\in\mathbb{P}^1(\F_p)} N(V^2_{p,\alpha};X/p^2)
\end{equation}
where we have counted such pairs $(R,R')$ by $R$ on the left and by $R'$ on the right.
Noting that
\begin{equation}\label{wpdecomp}
  N(\mathcal W_p;X) \;\; = \;\;
  N(\mathcal W_p\cap T_p(1^21);X) \; + \;
  N(\mathcal W_p\cap T_p(1^3);X) \; + \;
  N(pV_\Z;X)\,,
\end{equation}
together with (\ref{firstswitch}) and (\ref{pairscount2}), yields the following identity:
\begin{equation}\label{otherswitch}
N(\mathcal W_p\cap T_p(1^3); X) =
 \displaystyle\sum_{\alpha\in\mathbb{P}^1(\F_p)} N(V^2_{p,\alpha};X/p^2)
 \;\,-\,\displaystyle\sum_{\alpha\in\mathbb{P}^1(\F_p)}N(V_{p,\alpha};X/p^4) \;\,.
\end{equation}
Since we know that
$$
N(\ZZ_p;X)=N(\mathcal W_p;X)+N(T_p(1^3);X)-N(\mathcal W_p\cap T_p(1^3); X),
$$
we obtain (\ref{secondswitch}).
\end{proof}

For any squarefree $n\in\mathbb N$ and $\alpha\in\P^1(\Z/n\Z)$, let
$V_{n,\alpha}$ denote the set of all integral binary cubic forms $f\in
V_\Z$ such that the reduction of $f$ (mod~$n$) has a root at $\alpha$,
and $V_{n,\alpha}^2$ the set of all integral binary cubic forms $f\in
V_\Z$ such that the reduction of $f$ (mod~$p$) has at least a double
root at the reduction of $\alpha$ (mod $p$) for all primes $p$
dividing $n$. 

Then the above analysis generalizes in a straightforward way to squarefree
integers $n$ to give
\begin{eqnarray}\label{eqswitchidentity}
  \label{firstswitchforn}N(\mathcal W_n;X)&=&
  \displaystyle\sum_{\substack{{k,\ell,m\in\Z_{\geq0}}\\{k\ell m=n}\\\alpha\in\mathbb{P}^1(\mathbb{Z}/k\ell\mathbb{Z})}}\mu(\ell)N\left(V_{k\ell,\alpha};\frac{X}{k^2\ell^4m^4}\right)
  \,=\!\!\!\! \displaystyle\sum_{\substack{k,\ell\in\mathbb{Z}_{\geq0}\\k\ell|n\\ \alpha\in\mathbb{P}^1(\mathbb{Z}/k\ell\mathbb{Z})}}\mu(\ell)N\left(V_{k\ell,\alpha};\frac{Xk^2}{n^4}\right);\\[.2in]
\label{secondswitchforn}N(\mathcal Z_n;X)&=&\displaystyle\sum_{\substack{{k,\ell,m,q\in\Z_{\geq0}}\\{k\ell mq=n}\\\alpha\in\mathbb{P}^1(\mathbb{Z}/k\ell\mathbb{Z})}}\mu(\ell)N\left(V_{k,\alpha}\cap V^2_{\ell,\alpha}\cap T_q(1^3);\frac{X}{k^2\ell^2m^4}\right).
\end{eqnarray}




\subsection{Back to the sieve}

Let us define the error functions
$E^{(i)}_n(X)$ and $E^{(i)}_n(X/2,X)$ for squarefree $n$ by
\begin{equation}\label{mainerror}
\begin{array}{rcl}
  E^{(i)}_n(X) &=& \displaystyle{N(\mathcal W_n\cap V_\Z^{(i)};X)- \left({\gamma_1(n)}c_1^{(i)}X+\gamma_2(n)c_2^{(i)}X^{5/6}\right),}\vspace{.1in}\\
  E^{(i)}_n(X/2,X) &=& \displaystyle{N(\mathcal W_n\cap V_\Z^{(i)};X/2,X)- \left(\frac{\gamma_1(n)}{2}c_1^{(i)}X+\bigl(1-{2^{-5/6}}\bigr)\gamma_2(n)c_2^{(i)}X^{5/6}\right),}
\end{array}
\end{equation}
where $\gamma_1(n)$ and $\gamma_2(n)$ are defined by the conditions
$\gamma_1(p)+\mu_1(p)=\gamma_2(p)+\mu_2(p)=1$ for $n=p$ prime, and
$\gamma_1(n)=\prod_{p|n}\gamma_1(p)$ and
$\gamma_2(n)=\prod_{p|n}\gamma_2(p)$
for general squarefree $n$.
Returning to Equation (\ref{mainsieve}),  we write
\begin{align*}
  N(\U\cap V_\Z^{(i)};X/2,X)&= \displaystyle\sum_{n\in\mathbb{N}}\mu(n)N(\W_n\cap V_\Z^{(i)};X/2,X)\\
  &= \displaystyle\sum_{n\in\mathbb{N}}\mu(n) \left(\frac{\gamma_1(n)}{2}c_1^{(i)}X+\bigl(1-2^{-5/6}\bigr)\gamma_2(n)c_2^{(i)}X^{5/6}\right)+ \displaystyle\sum_{n\in\mathbb{N}}\mu(n)E_n^{(i)}(X/2,X)\\
  &= \displaystyle\frac{c_1^{(i)}X}{2\zeta(2)\zeta(3)} +\displaystyle\bigl(1-2^{-5/6}\bigr)\frac{c^{(i)}_2X^{5/6}}{\zeta(2)\zeta(5/3)} +\displaystyle\sum_{n\in\mathbb{N}}\mu(n)E_n^{(i)}(X/2,X).
\end{align*}
Thus to prove Theorem \ref{main1}, it is sufficient prove the estimate
\begin{equation}\label{errorsum}
\displaystyle\sum_{n\in\mathbb{N}}|E_n^{(i)}(X/2,X)|=O_\epsilon(X^{5/6-1/48+\epsilon}).
\end{equation}

Fix small numbers $\delta_1,\delta_2 > 0$ to be determined
later. We break up (\ref{errorsum}) into the three different ranges

\[0\leq n\leq X^{1/6\,-\,\delta_1},\,\, X^{1/6\,-\,\delta_1}\leq n\leq
X^{1/6\,+\,\delta_2},\textrm{ and } X^{1/6\,+\,\delta_2}\leq n\] and estimate
$\sum_n|E_n^{(i)}(X/2,X)|$ for $n$ in each range separately.

\subsection{The small and large ranges}

Suppose $n$ is a fixed positive integer. Let $k,\ell$ be positive
integers such that $k\ell\mid n$ and let $\alpha\in\P^1(\Z/k\ell\Z)$. Then, by
Theorem \ref{shincong}, there exist constants $c^{(i)}_1(\alpha)$ and
$c_2^{(i)}(\alpha)$ such that
\begin{equation}\label{eqveryimp}
N\left(V_{k\ell,\alpha}\cap V_\Z^{(i)};\frac{Xk^2}{2n^4},\frac{Xk^2}{n^4}\right)=
c_1^{(i)}(\alpha)\frac{Xk^2}{2n^4}+\bigl(1-2^{-5/6}\bigr)c_2^{(i)}(\alpha)\left(\frac{Xk^2}{n^4}\right)^{5/6}
+O_\epsilon\left(\frac{T_1^{1/3}X^{3/4+\epsilon}k^{3/2}}{n^3}\right)
\end{equation}
where, in the notation of Theorem \ref{shincong},
$T_1=T_1(k,\ell,\alpha)$ is an integer dividing $k\ell$ which depends
only on the lattice $V_{k\ell,\alpha}$.  Now, if a lattice $V_{k\ell,\alpha}$ satisfies $T_1(k,\ell,\alpha)=d$, then by the definition of $T_1$,
the image of $\alpha$ in $\P^1(\Z/d\Z)$ must be $0$. Hence, the number of choices for $\alpha$ is $O((k\ell/d)^{1+\epsilon})$. Since the total number of
$(k,\ell)$ such that $k\ell$ divides $n$ is $O(n^{\epsilon})$, we conclude that the number of lattices
$V_{k\ell,\alpha}$ satisfying $T_1(k,\ell,\alpha)=d$ is bounded by
$O(n^{1+\epsilon}/d)$. Therefore, from
(\ref{firstswitchforn}), (\ref{mainerror}), and (\ref{eqveryimp}),
we see that
\[|E_n^{(i)}(X/2,X)| = O_\epsilon\left(\displaystyle{\sum_{d|n}
\frac{n^{1+\epsilon}d^{1/3}X^{3/4+\epsilon}}{dn^{3/2}}}\right)=
O_{\epsilon}\left(\displaystyle\frac{X^{3/4\,+\,\epsilon}}{n^{1/2\,-\,\epsilon}}\right).\]
Summing over $n$, we conclude that
\begin{equation}\label{smallrange}
  \displaystyle\sum_{n=0}^{\,\,\,\,\,\,\,\,\,\,\,\,\,\,\, X^{1/6\,-\,\delta_1}}\!\!\!\!\!\!\!\!\!\!|E_n^{(i)}(X/2,X)| = O_{\epsilon}(X^{5/6\,-\,\delta_1/2\,+\,\epsilon}).
\end{equation}

From the definitions of $\gamma_1$ and $\gamma_2$, and from
(\ref{smtdenmax}), we have the estimates
\begin{equation}\label{gammaest}
\gamma_1(n)=O_{\epsilon}(n^{-2+\epsilon})\,\textrm{ and }\,\gamma_2(n)=O_{\epsilon}(n^{-5/3\,+\,\epsilon}).
\end{equation}
Let $q(n)$ denote the number of prime divisors of $n$. 
The next lemma follows from \cite[Lemmas~2.7 and 3.3]{BBP}:
\begin{lemma}\label{lemunifn}
For a square-free integer $n$, we have \[N(\ZZ_n;X)=O(3^{q(n)}X/n^2).\]
\end{lemma}
Thus we also have the estimate
\[N(\W_n;X)=O_{\epsilon}(X/n^{2-\epsilon}).\]
We deduce that

\[|E_n^{(i)}(X/2,X)|=O_{\epsilon}(X/n^{2-\epsilon})+O_{\epsilon}(X^{5/6}/n^{5/3-\epsilon}),\]
and summing up over $n$ we obtain
\begin{equation}\label{largerange}
\displaystyle\sum_{n\geq X^{1/6+\delta_2}} |E_n^{(i)}(X/2,X)| =
O_{\epsilon}(X^{5/6\,-\,\delta_2\,+\,\epsilon})+O_{\epsilon}(X^{13/18\,-\,2\delta_2/3\,+\,\epsilon}).
\end{equation}

In the next section, we estimate the sum of $|E_n^{(i)}(X/2,X)|$ over
the range $X^{1/6\,-\,\delta_1}\leq n\leq X^{1/6\,+\,\delta_2}.$

\subsection{An equidistribution argument}

We now concentrate on the middle range $X^{1/6\,-\,\delta_1}\leq n\leq
X^{1/6\,+\,\delta_2}.$ Let us write
\begin{equation}\label{eq941}
N(\W_n\cap V_\Z^{(i)};X)=\sum_{k\ell \mid n}\mu(m)S^{(i)}_{k\ell}(Xk^2/n^4),
\end{equation}
where
 \[S^{(i)}_n(X)=\displaystyle\sum_{\alpha\in\mathbb{P}^1(\mathbb{Z}/n\mathbb{Z})} N(V_{n,\alpha}\cap V_\Z^{(i)},X).\]

 In this section, we estimate $S^{(i)}_n(X)$, and then use
 (\ref{mainerror}) and (\ref{eq941}) to obtain a corresponding
 estimate on $E^{(i)}_n(X/2,X)$.  Given a form $f,$ let $w_n(f)$
 denote as before the number of roots in $\P^1(\Z/n\Z)$ of $f$
 (mod~$n$).  Then the number $S^{(i)}_n(X)$ counts the number of
 $\GL_2(\Z)$-equivalence classes of irreducible binary cubic forms in
 $V_\Z^{(i)}$, weighted by $w_n(f)$, having discriminant bounded by
 $X$. Thus
\begin{equation}\label{weighted count}
S^{(i)}_n(X)=\displaystyle\sum_{\substack{f\in\GL_2(\Z)\backslash V_\Z^{\irr}\\ |\Disc(f)|\leq X}} w_n(f).
\end{equation}


We now consider $w_n(f)$ as a function on $V_{\Z/n\Z}$ and bound its
Fourier transform pointwise. This in turn will allow us to count the number of binary cubic
forms $f$, weighted by $w_n(f)$, in small boxes (boxes with each side length at least $n^{3/4+\epsilon}$). We then can count
this weighted number of binary cubic forms in fundamental domains
using the ideas of Section 5, yielding the desired estimate for
$S^{(i)}_n(X)$, and therefore for $|E^{(i)}_n(X/2,X)|$.

Define $\widehat{V_{\Z}/nV_{\Z}}$ to be
the space of additive characters
$\chi:V_{\Z}/nV_{\Z}\to \mathbb{C}^{\times}.$ Then we define the Fourier
transform $\widehat{g} : \widehat{V_{\Z}/nV_{\Z}}\rightarrow \mathbb{C}$ of a
function $g:V_{\Z}/nV_{\Z} \rightarrow \mathbb{C}$ via
\[ \widehat{g}(\chi) := n^{-4}\displaystyle\sum_{v\in V_{\Z}/nV_{\Z}} g(v)\chi(v).\]
Fourier inversion then states that
$$g(v)=\displaystyle\sum_{\chi\in \widehat{V_{\Z}/nV_{\Z}}}\hat{g}(\chi)\bar{\chi}(v).$$

We focus now on computing $\widehat{w_n}(\chi).$ Assume first
that $n=p$ is prime.  We start with the trivial character which maps
all of $V_{\Z}/pV_{\Z}$ to 1, which we denote by $\m1$. Then
$$\widehat{w_p}(\m1) = p^{-4}\sum_{v\in V_{\Z}/pV_{\Z}}w_p(v) = 1+p^{-1}.$$
Now for any $\chi\neq \m1,$ we compute
\begin{equation}\label{eqfourierprimes}
\begin{array}{rcl}
  \widehat{w_p}(\chi) &=&  \displaystyle{p^{-4}\sum_{v\in V_{\Z}/pV_{\Z}}\chi(v)w_p(v)}\\[.05in]
  &=& \displaystyle{p^{-4}\displaystyle\sum_{\substack{v\,:\,\chi(v)=1 }}w_p(v) + p^{-4}\displaystyle\sum_{\substack{v\,:\,\chi(v)\neq 1}}w_p(v)\chi(v).}
\end{array}
\end{equation}
Since $\chi(v)=1$ for $p^3$ values of $v$ and $w_p(v)\leq 3$
for $v\neq 0$, we have the estimate
\begin{equation}\label{boundFourier1}
  \displaystyle\sum_{v\,:\,\chi(v)=1 }w_p(v)\leq 3(p^3-1)+(p+1)=3p^3+p-2.
\end{equation}
Because $w_p(\lambda v)=w_p(v)$ for any
$\lambda\in\mathbb{F}_p^{\times}$, we see that if $\chi(v)\neq 1$
then
$$\sum_{\lambda\in\mathbb{F}_p^{\times}}w_p(\lambda v)\chi(\lambda v)
= -w_p(v),$$ implying
\begin{equation}\label{eqboundFourier2}
  \displaystyle\sum_{\substack{v\,:\,\chi(v)\neq 1}}w_p(v)\chi(v) = -(p-1)^{-1}\displaystyle\sum_{\substack{v\,:\,\chi(v)\neq 1}} w_p(v).
\end{equation}
Combining (\ref{boundFourier1}) with (\ref{eqboundFourier2}), we see
that (\ref{eqfourierprimes}) implies that
\begin{equation}\label{boundprimesFourier}
\widehat{w_p}(\chi) \ll p^{-1}
\end{equation}
uniformly for $\chi\neq 0.$

Now let $n$ be a general squarefree integer. Then
$\widehat{V_{\Z}/nV_{\Z}}\cong\oplus_{p|n}\widehat{V_{\Z}/pV_{\Z}}$ and
$w_n(f)=\prod_{p|n}w_p(f)$. From this we conclude that $\widehat{w_n}(\chi)
= \prod_{p|n}\widehat{w_p}(\chi_p),$ where $\chi_p$ is the $p$-part of
$\chi.$ Using this and (\ref{boundprimesFourier}) implies that
\begin{equation}\label{boundFourier}
\widehat{w_n}(\chi) \ll \prod_{\substack{p|n\\ \chi_p\neq \m1}}p^{-1}
\end{equation}
and also
\begin{equation}\label{boundFourier2}
\widehat{w_n}(\m1) = \prod_{p\mid n}(1+p^{-1})=\sigma(n)/n,
\end{equation}
where $\sigma(n)=\sum_{d|n} d$ denotes as usual the sum-of-divisors function.

We now run through the argument in Section 5, counting integer binary
cubic forms $f$ weighted by $w_n(f)$.  Identically as in (\ref{avg}),
we have the following identity.
\begin{equation}\label{avg1}
  S^{(i)}_n(X) = \frac1{M_i}\int_{g\in N'(t)A'\Lambda}
  S^{(i)}_n(m,t,\lambda,X)t^{-2}
  dm\, d^\times t\,d^\times \lambda\,,
\end{equation}
where $$S^{(i)}_n(m,t,\lambda,X):=\!\!\!\!\!\displaystyle\sum_{x\in
  B(m,t,\lambda,X)}w_n(x).$$

To estimate $S^{(i)}_n(m,t,\lambda,X)$, we tile the set $B(m,t,\lambda,X)$
with boxes and count weighted integer cubic forms inside each box.
We have the following two lemmas.
\begin{lemma}\label{tile}
  Suppose $R$ is a region in $\R^4$ having volume $C_1$ and surface
  area $C_2$. Let $N$ be a positive integer. Then there exists a set
  $R'\subset R$ having volume equal to $C_1+O(N\cdot C_2)$ such that
  $R'$ can be tiled with $4$-dimensional boxes with all sides having length
  $N$.
\end{lemma}
\begin{proof}
  We first tile $\R^4$ with boxes having side length equal to $N$.
  Then we place $R$ inside $\R^4$ and take $R'$ to be the union of
  those boxes which lie entirely inside $R$. The region $R\setminus
  R'$ is within distance $N$ of the boundary of $R$. It is thus clear
  that the volume of $R'$ is equal to $C_1+O(N\cdot C_2)$.
\end{proof}

We now use equation (\ref{boundprimesFourier}) to establish the following
quantitative equidistribution statement for $w_n(f)$ inside boxes having
small sidelengths relative to $n$.

\begin{lemma}\label{countinboxes}
Let $\mathcal{B}\subset V_\R$ be a box with sides parallel to the
coordinate axes on $V_\R$  formed by the coefficients of the cubic
form $(a,b,c,d)$ such that each side has length $N\leq n$.
Then
$$\sum_{v\in\mathcal{B}\cap V_{\Z}}w_n(v) = \frac{\sigma(n)}{n}{\rm Vol}(\mathcal B) + O_{\epsilon}(n^{3+\epsilon}).$$
\end{lemma}
\begin{proof}
  Since each side length of $\mathcal{B}$ has side length at most $n$,
we can consider the set of lattice points in $\mathcal{B}$ as a subset
$\mathcal{B}_n$ of $V_{\Z}/nV_{\Z}.$
  We then use Fourier inversion to write
  \begin{align}
    \sum_{v\in\mathcal{B}\cap V_\Z}w_n(v)&= \sum_{v\in\mathcal{B}_n}\sum_{\chi\in\widehat{V_{\Z}/nV_{\Z}}}\widehat{w_n}(\chi)\bar{\chi}(v)\\
    &= N^4\widehat{w_n}(\m1) + \sum_{\substack{\chi\in\widehat{V_{\Z}/nV_{\Z}}\\
        \chi\neq \m1}}\widehat{w_n}(\chi)\sum_{v\in\mathcal{B}_n}\chi(-v)+O(N^3).
\end{align}

There is a $v_0\in V_{\Z}/nV_{\Z}$ such that
$\mathcal{B}_n=\{(a_1,a_2,a_3,a_4)+v_0\mid 0\leq a_1,a_2,a_3,a_4\leq
N-1\}.$ For each $\chi$, there are characters $\chi_i$, for $1\leq i\leq 4$,
such that $\chi(a_1,a_2,a_3,a_4)=\prod_{i=1}^4 \chi_i(a_i).$ Then, up to an error of $O(N^3)$,
we see that $\sum_{v\in\mathcal B_n}w_n(v)$ is equal to
\begin{align}\label{eqthemainsumte}
  N^4\widehat{w_n}(\m1) + \sum_{\substack{\chi\in\widehat{V_{\Z}/nV_{\Z}}\\
      \chi\neq \m1}}\widehat{w_n}(\chi)\sum_{v\in\mathcal{B}_n}\chi(-v)&=
  N^4\frac{\sigma(n)}{n} + \sum_{\substack{\chi\in\widehat{V_{\Z}/nV_{\Z}}\\ \chi\neq
      \m1}}\widehat{w_n}(\chi)\chi(-v_0)\prod_{i=1}^4\displaystyle\sum_{a_i=0}^{N-1}\chi_i(-a_i).
\end{align}
We estimate the sum over each $\chi\neq \m1$ separately. By
(\ref{boundFourier}), we know $|\widehat{w_n}(\chi)|\ll
\displaystyle\prod_{\substack{p\mid n\\ \chi_p\neq \m1}}p^{-1}.$ Now, for a character
$\psi$ of $\Z/n\Z,$ we define $A_N(\psi)$ by
$$A_N(\psi):=\sum_{a=0}^{N-1}
\psi(a)=\begin{cases}
  N & \psi=\m1\\
  \displaystyle\frac{1-\psi(N)}{1-\psi(1)} & \psi\neq \m1
\end{cases}$$ and then define $A_N(\chi):=\prod_{i=1}^4 A_N(\chi_i).$
This implies that $\displaystyle\sum_{\psi\in \widehat{\Z/n\Z}}
|A_N(\psi)|\ll \displaystyle\sum_{k=1}^n\frac{n}{k}\ll n\log\,n$.

We now estimate the right hand side of (\ref{eqthemainsumte}) as
follows:
\begin{align*}
  N^4\frac{\sigma(n)}{n} + \sum_{\substack{\chi\in\widehat{V_{\Z}/nV_{\Z}}\\ \chi\neq
      \m1}}\widehat{w_n}(\chi)\chi(-v)\prod_{i=1}^4\displaystyle\sum_{a_i=0}^{N-1}\chi_i(-a_i)&=
  N^4\frac{\sigma(n)}{n} + O\Bigl( \sum_{\substack{\chi\in\widehat{V_{\Z}/nV_{\Z}}\\ \chi\neq \m1}}|A_N(\chi )\widehat{w_n}(\chi)|\Bigr)\\
  &= N^4\frac{\sigma(n)}{n} + O_{\epsilon}(n^{3+\epsilon}),\\
\end{align*}
where the last bound follows from
\begin{align*}
\sum_{\substack{\chi\in\widehat{V_{\Z}/nV_{\Z}}\\ \chi\neq
    \m1}}|A_N(\chi)\widehat{w_n}(\chi)| &\leq
\sum_{\substack{d|n\\1<d}}d^{-1}\sum_{\substack{\chi\\\chi_p\neq \m1\,\forall
    p|d\\\chi_p= \m1\,\forall p\nmid d}}|A_N(\chi)|\\
&\leq \sum_{\substack{d|n\\1<d}}d^{-1}\Bigl(\bigl(\sum_{\psi\in\widehat{\Z/d\Z}}|A_N(\psi)|\bigr)^4 - N^4\Bigr)\\
&\leq \sum_{\substack{d|n\\1<d}}d^{-1}\Bigl((N+O(d\log d))^4 - N^4\Bigr)\\
&\leq \sum_{\substack{d|n\\1<d}}O_{\epsilon}(\max(d,N)^{3+\epsilon})\\
&\leq O_{\epsilon}(n^{3+\epsilon}).
\end{align*}
This completes the proof of the lemma.
\end{proof}

We now estimate $S^{(i)}_n(m,t,\lambda,X)$ for $|m|<1/2$.
First, tile $B(m,t,\lambda,X)'\subset B(m,t,\lambda,X)$ with boxes
using Lemma~\ref{tile}.  Note that the region $B(m,t,\lambda,X)$ is
obtained by acting on the region $B(1,1,1,\frac{X}{\lambda^4})$ by
$m\cdot t\cdot\lambda\in\GL_2(\R)$. So the surface area of
$B(m,t,\lambda,X)$ is $O(\lambda^3t^3)$.
We thus have
\begin{equation}\label{eq976}
S^{(i)}_n(m,t,\lambda,X)=\frac{\sigma(n)}{n}{\rm Vol}(B(m,t,\lambda,X))+O_{\epsilon}\left(\frac{n^{3+\epsilon}\lambda^4}{N^4}\right)+O(\lambda^3t^3N),
\end{equation}
where the first error term comes from Lemma \ref{countinboxes} and the
second comes from Lemma \ref{tile}. We optimize by picking
$N=\lambda^{1/5}t^{-3/5}n^{3/5}$. Using (\ref{eq976}), as in Section 5, we
evaluate the right hand side of (\ref{avg1}) to obtain
\begin{equation}\label{seq}
S^{(i)}_n(X)=\frac{\sigma(n)}{n}c_1^{(i)}X+O_{\epsilon}(n^{3+\epsilon}+X^{5/6}n^{1/2}).
\end{equation}
Using (\ref{firstswitchforn}), (\ref{mainerror}),
(\ref{gammaest}), 
and (\ref{seq}) we
finally arrive at the bound
$$|E^{(i)}_n(X)| \leq \gamma_2(n)X^{5/6} + O_{\epsilon}(n^{\epsilon})\Bigl(\sum_{\substack{ k,\ell\in\Z\\ k\ell|n}} (k\ell)^3 + \frac{X^{5/6}k^{5/3}}{n^{17/6}}\Bigr).$$
Therefore, we have
$$|E^{(i)}_n(X)|= O_{\epsilon}(n^{\epsilon})\Bigr(\frac{X^{5/6}}{n^{7/6}} + n^3\Bigr)$$
implying
\begin{equation}\label{middlerange}
  \displaystyle\sum_{n=X^{1/6-\delta_1}}^{X^{1/6+\delta_2}}
  |E^{(i)}_n(X)|=O_{\epsilon} \bigl(X^{29/36 + {\delta_1}/{6}+\epsilon}
  + X^{2/3 +4\delta_2+\epsilon}\bigr) .
\end{equation}
This also implies the estimate
\begin{equation}\label{middleranged}
  \displaystyle\sum_{n=X^{1/6-\delta_1}}^{X^{1/6+\delta_2}} |E^{(i)}_n(X/2,X)| =O_{\epsilon}\bigl(X^{29/36 + {\delta_1}/{6}+\epsilon} + X^{2/3 +4\delta_2+\epsilon}\bigr).
\end{equation}

\subsection{Putting it together}

We combine (\ref{smallrange}), (\ref{largerange}) and
(\ref{middleranged}) to obtain
$$
\sum_{n\in\mathbb{Z}} |E^{(i)}_n(X/2,X)| \ll_{\epsilon} 
X^{5/6\,-\,\delta_1/2\,+\,\epsilon}+X^{5/6-\delta_2+\epsilon}+X^{13/18 -2\delta_2/3+\epsilon}+X^{29/36\,+\,\delta_1/6\,+\,\epsilon} + X^{2/3\,+\,4\delta_2\,+\,\epsilon}.
$$ 
We
optimize by picking $\delta_1=\frac{1}{24}$ and
$\delta_2=\frac{1}{30}$ to get
$$\sum_{n\in\mathbb{Z}} |E^{(i)}_n(X/2,X)|\ll_{\epsilon} X^{5/6\,-\,1/48\,+\,\epsilon},$$
which proves Theorem \ref{main1}.

Finally, note that the values of $\mu_1(\sigma,p)$ and
$\mu_2(\sigma,p)$ that we list in Table \ref{tab1} are the same as the
values of $C_{p,\alpha_p}$ and $K_{p,\alpha_p}$, respectively, in
\cite[Equation (5.1)]{Roberts}.  We thus also obtain Roberts' refined
conjecture (see \cite[Section 5]{Roberts}); the proof is identical to
the proof of Theorem \ref{main1}.

\subsection{Proof of Theorem \ref{main2}}
The proof of Theorem \ref{main2} is very similar to that of Theorem \ref{main1}.
This time, we
define the error function
$F^{(i)}_n(X/2,X)$ for squarefree $n$ by
\begin{equation}\label{mainerror1}
\begin{array}{rcl}
  F^{(i)}_n(X/2,X) &=& \displaystyle{N(\mathcal Z_n\cap V_\Z^{(i)};X/2,X)- \left(\frac{\gamma'_1(n)}{2}c_1^{(i)}X+\bigr(1-{2^{-5/6}}\bigr)\gamma'_2(n)c_2^{(i)}X^{5/6}\right),}
\end{array}
\end{equation}
where $\gamma'_1(n)$ and $\gamma'_2(n)$ are defined by the conditions
$\gamma'_1(p)+\mu'_1(p)=\gamma'_2(p)+\mu'_2(p)=1$ for $n=p$ prime, and
$\gamma'_1(n)=\prod_{p|n}\gamma'_1(p)$ and
$\gamma'_2(n)=\prod_{p|n}\gamma'_2(p)$
for general squarefree $n$.
We can write
\begin{align*}
  N(\V\cap V_\Z^{(i)};X/2,X)&= \displaystyle\sum_{n\in\mathbb{N}}\mu(n)N(\ZZ_n\cap V_\Z^{(i)};X/2,X)\\
  &= \displaystyle\sum_{n\in\mathbb{N}}\mu(n) \left(\frac{\gamma'_1(n)}{2}c_1^{(i)}X+\bigl(1-{2^{-5/6}}\bigr)\gamma'_2(n)c_2^{(i)}X^{5/6}\right)+ \displaystyle\sum_{n\in\mathbb{N}}\mu(n)F_n^{(i)}(X/2,X)\\
  &= \displaystyle\frac{c_1^{(i)}X}{2\zeta(2)\zeta(3)} +\displaystyle\bigl(1-{2^{-5/6}}\bigr)\frac{c^{(i)}_2X^{5/6}}{\zeta(2)\zeta(5/3)} +\displaystyle\sum_{n\in\mathbb{N}}\mu(n)F_n^{(i)}(X/2,X).\\
\end{align*}

\vspace{-.2in}
\noindent

Thus, to prove Theorem~\ref{main2}, it is sufficient prove the estimate
\begin{equation}\label{errorsum2}
\displaystyle\sum_{n\in\mathbb{N}}|F_{n}^{(i)}(X/2,X)|=O_\epsilon(X^{5/6-1/48+\epsilon}).
\end{equation}
Let $\delta_1,\delta_2 > 0$ be as in the previous subsection.
Again, we break up (\ref{errorsum2}) into the three different ranges
\[0\leq n\leq X^{1/6\,-\,\delta_1},\,\, X^{1/6\,-\,\delta_1}\leq n\leq
X^{1/6\,+\,\delta_2},\textrm{ and } X^{1/6\,+\,\delta_2}\leq n\] and estimate
$\displaystyle\sum_{n}|F_{n}^{(i)}(X/2,X)|$ for $n$ in each range separately.


In Equation (\ref{secondswitchforn}), we write $N(\mathcal Z_n\cap V_\Z^{(i)};X/2,X)$
as a sum over positive integers $k,\ell,m,q$ with $k\ell m q=n$. 
Let $k,\ell\in\Z_{>0}$ such that $k\ell |n$. Then, for
$\alpha\in\P^1(\Z/k\ell\Z)$, we may write $V_{k,\alpha}\cap
V^2_{\ell,\alpha}\cap T_q(1^3)$ as a union of $O(q^2)$ translates of lattices, each
of which has index $k\ell^2q^4$ in $V_\Z$ and is defined via
congruence conditions modulo~$k\ell q$.
The remark following Theorem \ref{shincong} implies that for each of these lattice-translates $\L$
there exist constants $c^{(i)}_1(\L)$ and
$c_2^{(i)}(\L)$ such that
\begin{equation}\label{eqveryimp1}
N\left(\L;\frac{Xk^2\ell^2q^4}{2n^4},\frac{Xk^2\ell^2q^4}{n^4}\right)=
c_1^{(i)}(\L)\frac{Xk^2\ell^2q^4}{2n^4}+\bigl(1-{2^{-5/6}}\bigr)c_2^{(i)}(\L)\left(\frac{Xk^2\ell^2q^4}{n^4}\right)^{5/6}
+O_\epsilon\left(\frac{X^{3/4}k^{3/2}\ell^{1/2}}{n^3}\right).
\end{equation}
Since there are $O_\epsilon(n^\epsilon k\ell q^2)$ such lattices, we see that
\[|F_{n}^{(i)}(X/2,X)| = O_\epsilon\left(\sum_{n=kn_1}n^\epsilon\frac{X^{3/4}}{k^{1/2}\ell^{3/2}m^3q}\right)=O_\epsilon\left(\frac{X^{3/4}}{n^{1/2-\epsilon}}\right).\]
Summing over $n=kn_1$ in the small range, we conclude that
\begin{equation}\label{smallrange1}
  \displaystyle\sum_{n=1}^{\,\,\,\,\,\,\,\,\,\,X^{\frac16\,-\,\delta_1}}|F_{n}^{(i)}(X/2,X)| = O_{\epsilon}(X^{5/6\,-\,\delta_1/2\,+\,\epsilon}).
\end{equation}

As in Section 9.3, we may use Lemma \ref{lemunifn} to estimate
$\sum_{n}|F_{n}^{(i)}(X/2,X)|$ over $n$ lying in the
large range:
\begin{equation}\label{largerange1}
  \displaystyle\sum_{n\geq X^{\frac16+\delta_2}}|F_{n}^{(i)}(X/2,X)| = O_{\epsilon}(X^{5/6\,-\,\delta_2\,+\,\epsilon})+O_\epsilon(X^{13/18-2\delta_2/3+\epsilon}).
\end{equation}

We now consider the middle range. Fix $k,\ell,q,m$ such that $k\ell
qm=n$. For $\beta\in\P^1(\Z/\ell\Z)$, we may write
$V^2_{\ell,\beta}\cap T_p(1^3)$ as a union of $O(p^2\ell^2)$ translates
of $p\ell V_\Z$. Let $\L$ be one of them. Identically to Section~9.4, using equation \eqref{seq} we have:
$$\sum_{\alpha\in\P^1(\Z/k\Z)}N\left(V_{\Z}^{(i)}\cap V_{k,\alpha}\cap\L;\frac{X}{k^2\ell^2m^4}\right)=c^{(i)}(\L)X+O_\epsilon\left(k^{3+\epsilon}+k^{1/2}\left(\frac{X}{k^2\ell^6m^4p^4}\right)^{\frac56}\right),$$
where $c(\L)$ is some explicit constant.
It follows, just as in Section 9.4, that
\begin{equation}\label{middleranged1}
  \displaystyle\sum_{n=X^{1/6-\delta_1}}^{X^{1/6+\delta_2}} |F^{(i)}_n(X/2,X)|\ll_{\epsilon} X^{29/36 + \frac{\delta_1}{6}+\epsilon} + X^{2/3 +4\delta_2+\epsilon}.
\end{equation}

Finally, note that 
\begin{equation}\label{trDH1}
\begin{array}{ccc}
  \displaystyle\sum_{0<\Disc(K_2)<X} 1  &=  \displaystyle\frac{3}{\pi^2}\cdot X +O(X^{\frac12}); \\[.25in]
  \displaystyle\sum_{-X<\Disc(K_2)<0} 1  &= \displaystyle\frac{3}{\pi^2}\cdot X +O(X^{\frac12}). \\[.25in]
\end{array}
\end{equation}
Theorem \ref{main2} may now be deduced from Equations
(\ref{smallrange1}), (\ref{largerange1}), and (\ref{middleranged1}) (together with (\ref{l4eq}) and (\ref{trDH1})) just
as Theorem~\ref{main1} was deduced in Section 9.5 from Equations
(\ref{smallrange}), (\ref{largerange}), and (\ref{middleranged}).
\subsection{Another simultaneous generalization}

In this subsection, we prove Theorem \ref{gensigma1}.\\[.15in] {\bf
  Proof of Theorem \ref{gensigma1}:} Let $p$ be a fixed finite prime.
If $R\in\Sigma_p$ is a cubic ring over $\Z_p$,
then we define $V(R)\subset V_\Z$ to be the set of
all integer binary cubic forms $f$ such that the corresponding
cubic ring $C$ satisfies $C\otimes\Z_p\cong R$. As
in Section 7, we define $\mu_1(R,p)$ and $\mu_2(R,p)$ to be such that
$$N(V(R)\cap V^{(i)}_\Z;X)=\mu_1(R,p)c_1^{(i)}X+\mu_2(R,p)c_2^{(i)}
X^{5/6}+O_{\epsilon}(X^{3/4+\epsilon}).$$
Using the same techniques as in the proofs of Theorems \ref{main1} and \ref{main2}, we
have
\begin{equation}
\begin{array}{rcl}
N(\Sigma;X)\!\!\!&=&\!\!\!
\Bigl(\frac12\displaystyle\sum_{R\in\Sigma_\infty}
\frac1{|\Aut_\R(R)|}\Bigr)\cdot
\prod_p\Bigl(\displaystyle\sum_{R\in\Sigma_p}\mu_1(R,p)\Bigr) \cdot \zeta(2)\cdot
X \vspace{.1in}\\[.1in] &+&\,\,
\Bigl(\displaystyle\sum_{R\in\Sigma_\infty} c_2(R)\Bigr)\cdot
\prod_p\Bigl(\displaystyle\sum_{R\in\Sigma_p}\mu_2(R,p)\Bigr) \cdot
X^{5/6}\,\,\vspace{.1in}\\ &+&\,\,O_{\epsilon}(X^{5/6-1/48+\epsilon}).
\end{array}\end{equation}

We now prove the following lemma:
\begin{lemma} With notation as above, we have
  \[\mu_2(R,p)=(1-p^{-2})(1-p^{-1/3})\Bigl(\frac1{\Disc_p(R)}\cdot\frac1{|\Aut(R)|}\int_{(R/\Z_p)^{{\rm
        Prim}}}i(x)^{2/3}dx\Bigr).\]
\end{lemma}
\begin{proof}
  Fix a form $f\in V_{\Z_p}$ corresponding to $R$.
  Let $m$ be a positive integer such that $p^m$ is larger than
  $\Disc_p(R)$, so that in particular $\Disc(f)\not\equiv 0$ (mod
  $p^m$).  Let ${F}=\{f_1,f_2,\dots,f_r\}$ be the
  $\GL_2(\Z/p^m\Z)$-orbit of the reduction of $f$ (mod $p^m$).
  By the slicing techniques of Section~6, as used in the proof of
Theorem~\ref{shincong}, we have
$$\mu_2(R,p)=p^{-3m}\cdot
\displaystyle \Biggl.{\frac{\displaystyle\sum_{i=1}^r
\displaystyle\sum_{a\equiv a(f_i)} a^{-s}}{\displaystyle{\sum_{a\neq
    0}a^{-s}}}}
\Biggr|_{s=1/3},$$
where $a(f_i)$ denotes the $x^3$-coefficient of $f_i$ and the
congruences are taken modulo $p^m$. Since ${F}$ is
$\GL_2(\Z/p^m\Z)$-invariant, every value of $a(f_i)$ with the same
$p$-adic valuation occurs equally often in ${F}$. Therefore,
we have
\begin{equation}\label{mu2form}
\mu_2(R,p)=(1-p^{-1/3})p^{-3m}\displaystyle\sum_{i=1}^r \begin{cases}
  \displaystyle\frac{p^{1-m}|a(f_i)|_p^{-2/3}}{p-1} & \mbox{if } a(f_i)\not\equiv 0 \!\!\!\!\pmod{p^m}\vspace{.1in}\\
  \displaystyle\frac{p^{-m/3}}{1-p^{-1/3}} & \mbox{if } a(f_i)\equiv 0 \!\!\!\!\pmod{p^m}. \\
\end{cases}
\end{equation}

The group
$\GL_2(\Z_p)$ acts on $f$ in the natural way. Normalizing the Haar
measure so as to give $\GL_2(\Z_p)$ measure 1, we may rewrite
(\ref{mu2form}) as
$$\mu_2(R,p) = \frac{(1-p^{-2})(1-p^{-1/3})}
{|\Aut_{\GL_2(\Z/p^m\Z)}(f)|}\cdot\int_{\GL_2(\Z_p)}|a(g\cdot
f)|_p^{-2/3}dg.$$ 
The above equality holds since we are in the first case of (\ref{mu2form}) when $m$ is sufficiently large, and
$$
r=\#F=\frac{|\GL_2(\Z/p^m\Z)|}{|\Aut_{\GL_2(\Z/p^m\Z)}(f)|}=\frac{p^{4m}(1-p^{-2})(1-p^{-1})}{|\Aut_{\GL_2(\Z/p^m\Z)}(f)|}.
$$
Now, by computing the measure of $\GL_2(\Z_p)\cdot
f$ using two different methods, we obtain $$|\Aut_{\GL_2(\Z/p^m\Z)}(f)| =
|\Aut_{\GL_2(\Z_p)}(f)|\cdot \Disc_p(f).$$ The first method is by
splitting $\GL_2(\Z_p)\cdot f$ into $p^m\cdot V_{\Z_p}$ cosets. The number
of such cosets is exactly $|\GL_2(\Z/p^m\Z)|\cdot
|\Aut_{\GL_2(\Z/p^m\Z)}(f)|^{-1}$.  The second method is by integrating
over the group, and using that the left invariant measure on
$V_{\Z_p}$ is $|\Disc(v)|^{-1}dv$ and the map $g\rightarrow g\cdot f$
is a $|\Aut_{\GL_2(\Z_p)}(f)|$-to-1 cover.

We thus have
$$\mu_2(R,p) = \frac{(1-p^{-2})(1-p^{-1/3})}{\Disc_p(f)\cdot|\Aut_{\GL_2(\Z_p)}(f)|}\cdot\int_{\GL_2(\Z_p)}|a(g\cdot f)|_p^{-2/3}dg.$$
Note that $a(g\cdot f)=f(v_0\cdot g)$ where
$v_0=(1,0)\in\Z_p\times\Z_p$.  Therefore, we have
$$\int_{\GL_2(\Z_p)}|a(g\cdot f)|_p^{-2/3}dg=\int_{(\Z_p^2)^{{\rm Prim}}}|f(v)|_p^{-2/3}dv,$$
where $dv$ is normalized to have measure $1$ on $(\Z_p^2)^{{\rm
    Prim}}$.

From the correspondence in Section 2, we see that the set
$(\Z_p^2)^{{\rm Prim}}$ corresponds to $(R/\Z_p)^{{\rm Prim}}$ and
that for $v\in(\Z_p^2)^{{\rm Prim}}$ corresponding to $x\in R$,
the value of $f(v)$ is equal to the index of $\Z[x]$ in $R$.
The lemma follows.
\end{proof}

Theorem \ref{gensigma1} now follows from Theorem~\ref{gensigma2}
and the above lemma. $\Box$



\subsection*{Acknowledgments}

We thank Mohammad Bardestani, Karim Belabas, Andrew Granville, Piper
Harris, Carl Pomerance, Peter Sarnak, Christopher Skinner, Frank
Thorne, Ila Varma, Melanie Wood, and the anonymous referees for
helpful comments on earlier versions of this manuscript. We are also
grateful to Boris Alexeev and Sucharit Sarkar for helping us compute
the precise values of the second main terms.




\end{document}